\documentclass[12pt]{amsart}
\usepackage{amsfonts, amsmath, amsthm}
\usepackage{graphicx}
\usepackage{fourier}
\usepackage{geometry}
\usepackage{color}

\newtheorem{pro}{Proposition}[section]
\newtheorem{thm}[pro]{Theorem}
\newtheorem{lem}[pro]{Lemma}

\newtheorem{question}[pro]{Question}

\newtheorem{clm}[pro]{Claim}

\newtheorem{ass}{Assertion}

\newtheorem{cor}[pro]{Corollary}

\theoremstyle{definition}
\newtheorem{prob}[pro]{Problem}
\newtheorem{dfn}[pro]{Definition}
\newtheorem{notation}[pro]{Notation}
\newtheorem{dfns}[pro]{Definitions}
\newtheorem{example}[pro]{Example}
\newtheorem{rmkk}[pro]{Remark}

\theoremstyle{remark}
\newtheorem*{rmk}{Remark}

\newcommand{\fff}{{\ensuremath{f_K}}}

\newcommand{\scc}{simple closed curve}
\newcommand{\s}{\Sigma}

\newcommand{\sft}{{swallow follow torus}}

\newcommand{\del}{\partial}

\newcommand{\gr}{\mbox{gr}}

\def\qbc{\ensuremath{Q^{(b,c)}_{\beta,\gamma}}}

\def\genbd[#1]{\ensuremath{6-2#1}}

\title{The growth rate of the tunnel number of m-small knots}

\date{\today}
\address{Department of Mathematics, Nara Women's University
Kitauoya Nishimachi, Nara 630-8506, Japan}
\address{Department of mathematical Sciences, University of
Arkansas, Fayetteville, AR 72701}
\email{tsuyoshi@cc.nara-wu.ac.jp}
\email{yoav@uark.edu}
\author{Tsuyoshi Kobayashi}
\author{Yo'av Rieck}
\thanks{The first names author was supported by Grant-in-Aid for scientific
  research, JSPS grant number 25400091.  
  This work was partially supported by a grant from the Simons Foundation (283495  to Yo'av Rieck).}

\begin{document}


\date{\today}%
%\dedicatory{}%
%\commby{}%
% ----------------------------------------------------------------
%% 
\begin{abstract}
In~\cite{kobayashi-rieck-growth-rate} the authors defined the {\it growth rate of the tunnel number of knots},
an invariant that measures that asymptotic behavior of the tunnel number under connected sum.  
In this paper we calculate the  growth rate of the tunnel number of m-small knots in terms of their bridge indices.
\end{abstract}
\maketitle
% ----------------------------------------------------------------
%
%\textcolor{blue}{Using the command genbd for the genus bound: \genbd[g] or \genbd[g(\s)].
%Works with or without math mode.  For now it is set to {\bf 6} - 2().}

\tableofcontents

\part{Introduction and background material}

\section{Introduction}
\label{sec:introduction}

Let $M$ be a compact connected orientable 3-manifold and $K \subset M$ a knot.  $K$ is
called {\it admissible} if $g(E(K)) > g(M)$ and inadmissible otherwise (throughout this
paper  $E(\cdot)$ denotes knot exterior and 
$g(\cdot)$ denotes the Heegaard genus; see
Section~\ref{sec:background} for these and other basic definitions).  Let $nK$ denote the connected
sum of $n$ copies of $K$.  
In \cite{kobayashi-rieck-growth-rate} the authors
defined the {\it growth rate} of the tunnel number of $K$ to be:
$$\gr_t(K) = \limsup_{n \to \infty} \frac{g(E(nK)) - ng(E(K)) + n-1}{n-1}$$   
The main result of \cite{kobayashi-rieck-growth-rate} shows  that if $K$ is admissible then
$\gr_t(K) <1$, and $\gr_t(K) = 1$ otherwise.  This concept was the key to constructing a counterexample
to Morimoto's Conjecture~\cite{KobayashiRieckMC} and~\cite{KobayashiRieckAnnouncement}.
Unless explicitly stated otherwise, 
all knots considered are assumed to be admissible (note that this is always the case for knots in the three sphere $S^3$).

In this paper we continue our investigation of the growth rate of the tunnel
number.  In Part~2 we give an upper bound on the growth rate of admissible
knots (this is an improvement of the bound given in
\cite{kobayashi-rieck-growth-rate}), and in Part~3 we obtain a lower bound on 
the growth rate of admissible m-small knots (a knot is called  
{\it m-small} if its meridian is not a boundary slope of an essential surface).  With this we obtain
an exact calculation of the growth rate of m-small knots.
Before stating this result we define the following notation that will
be used extensively throughout the paper:

\begin{notation}\label{notation:bridge indices}
Let $K \subset M$ be an admissible knot.  We denote $g(E(K)) - g(M)$ 
by $g$ and for $i=1,\dots,g$ we denote the bridge index of $K$ with respect to
Heegaard surfaces of genus $g(E(K)) - i$ by $b_i^*$.  
That is, $b_i^*$ is the minimal integer so that $K$ admits a $b_i^*$ bridge position with respect to some 
Heegaard surface of $M$ of genus $g(E(K)) - i$; we call such a decomposition a {\it
$(g(E(K)) - i, b^{*}_{i})$ decomposition}.    
Note that for a knot $K \subset S^{3}$ we have that $g = g(E(K))$, $b_{g}^*(K)$ is the bridge index of $K$, and $b_{g-1}^{*}(K)$
is the torus bridge index of $K$.  
We note that, for any knot $K \subset M$,
$b_{i}^{*}$ forms an increasing 
sequence of positive integers: 
$0 < b_{1}^{*} < \cdots < b_{g}^{*}$.
To see this, fix $i \geq 1$ and let $\s$ be a Heegaard surface that realizes the
bridge index $b_i^*$, that is, $\s$ is a genus $g(E(K)) - i$ Heegaard surface for $M$
with respect to which $K$ has bridge index $b_i^*$.  By tubing $\s$ once (see Definition~\ref{dfn:tubing})
we obtain a Heegaard surface of genus $g(E(K)) - (i-1)$ that realizes a $(g(E(K)) - (i-1),b_{i}^*-1)$
decomposition for $K$.  This shows that $b_{i-1}^* \leq b_i^*-1$.  
\end{notation}

We are now ready to state:

\begin{thm}
  \label{thm:main}
Let $M$ be a compact connected orientable 3-manifold and
$K \subset M$ be an admissible knot.  
Then $\gr_t(K) \leq  \min_{i=1,\dots,g}
\big\{1- \frac{i}{b_i^*}\big\}$.  If, in addition, $K$ is m-small then equality holds:
$$\gr_t(K) = \min_{i=1,\dots, g} \bigg\{1- \frac{i}{b_i^*}\bigg\}$$
Moreover, for m-small knots the limit of \ $\frac{g(E(nK)) - ng(E(K)) +
  n-1}{n-1}$ exists.
\end{thm}

As noted in Notation~\ref{notation:bridge indices}, the indices $b_i^*$ form an increasing series of
positive integers.  It follows that $b_i^* \geq i$; moreover, $b_i^* = i$ implies that $b_1^* = 1$.
Applying this to an index $i$ that realizes that the equality 
$\gr_t(K) = 1- \frac{i}{b_i^*}$ we obtain the following simple and useful consequence of 
Theorem~\ref{thm:main} that strengthens the main result of~\cite{kobayashi-rieck-growth-rate}
in the case of m-small knots:

\begin{cor}
\label{cor:bounds}
If $K \subset M$ is an admissible m-small knot, then
$$0 \leq \gr(K) <1$$
Moreover, $\gr(K) = 0$ if and only if $b_1^* = 1$.
\end{cor}

There are several results about the spectrum of the growth rate and we summarize them here. 
It is well known that there exist manifolds $M$ that admit minimal genus Heegaard splittings 
$\Sigma$ of genus at least 2 and of  Hempel distance at least 3.
We fix such $M$ and $\Sigma$ and for simplicity we assume that $M$ is closed.   
Let $C$ be a handlebody obtained by cutting $M$ along $\Sigma$ and 
$K$ a core of $C$, that is, $K$ is a core of a solid torus obtained by cutting $C$
along appropriately chosen meridian disks.
Then $\Sigma$ is a Heegaard surface for $E(K)$; it follows that  $K$ is inadmissible. 
Clearly, the Hempel distance 
does not go down after drilling $K$.  Hence the Hempel distance of $\Sigma \subset E(K)$ 
is at least 3.  It is a well known consequence of Thurston-Perelman's Geometrization Theorem
that manifolds that admit a Heegaard surface of genus at least 2 and Hempel
distance at least 3 are hyperbolic.  Thus $K \subset M$ is a hypebolic
knot in a hyperbolic manifold.   As mentioned above, the growth rate
of inadmissible knots is 1.
This proves existence of hyperbolic knots in hyperbolic manifolds
with growth rate 1.
It was shown in \cite{kobayashi-rieck-growth-rate} that torus knots and 
2-bridge knots have growth rate 0.  
Kobayashi and Saito~\cite{KobayashiSaito} constructed knots with growth rate $-1/2$.
Theorem~\ref{thm:main} enables us to calculate the growth rate  of the knots constructed 
by Morimoto, Sakuma and Yokota in \cite{skuma-morimoto-yokata}
(perhaps with finitely many exceptions), which we denote by $K_{MSY}$.
We explain this here.   The knots $K_{MSY}$ enjoy the following properties:

\begin{enumerate}
\item  $K_{MSY}$ are hyperbolic and m-small: 
this was announced by by Morimoto in a preprint available at~\cite{morimoto3}.

\item $g(E(K_{MSY})) = 2$: this was proved by 
Morimoto, Sakuma, and Yokota~\cite{skuma-morimoto-yokata}.

\item  $b_1^{*}(K_{MSY}) = 2$ (in other words, the torus bridge index of $K_{MSY}$ is 2):
it was shown in \cite{skuma-morimoto-yokata} that $b_{1}^{*} > 1$, and 
it is easy to observe that $b_{1}^{*} \le 2$ (see, for example, 
\cite{kobayashi-rieck-growth-rate}).

\item  $b_2^{*}(K_{MSY}) \geq 4$ (in other words, the bridge index of $K_{MSY}$ is at least 4):
since $b_{2}^{*}(K_{MSY}) > b_{1}^{*}(K_{MSY})$,  we only need to exclude the possibility
$b_{2}^{*}(K_{MSY}) = 3$.  Assume for a contradiction that $b_{2}^{*}(K_{MSY}) = 3$.
Then $K_{MSY}$ is a tunnel number 1, 3-bridge knot.
Kim~\cite{kim}  proved that every tunnel number 1,
3-bridge knot has torus bridge index 1, contradicting the previous point.
Recently R Bowman, S Taylor and A Zupan~\cite{BowmanTaylorZupan} showed that
$ b_{2}^{*}(K_{MSY}) = 7$ for all but finitely many of the knots $K_{MSY}$
(see Remak~\ref{rmk:RecentWork}).
\end{enumerate}

Using these facts, Theorem~\ref{thm:main} implies that $\gr_t(K_{MSY}) = 1/2$.  This is the first example
of knots with growth rate in the open interval $(0,1)$ and provides partial answer to questions
posed in~\cite{kobayashi-rieck-growth-rate}.
In summary we have the following; we emphasize that only~(4) is new: 

\begin{cor}
The following  holds:
\begin{enumerate}
\item There exist hyperbolic knots in hyperbolic manifolds with growth rate 1.
\item There exist hyperbolic knots in $S^{3}$ with growth rate 0.
\item There exist knots in $S^{3}$ with growth rate -1/2.
\item There exist hyperbolic knots in $S^{3}$ with growth rate 1/2.
\end{enumerate}
\end{cor}

\begin{rmkk}
In~\cite{BakerKobayashiRieck} K Baker and the authors use Theorem~\ref{thm:main}
to show that for any $\epsilon > 0$ there exists a hyperbolic knot $K \subset S^3$
with $1 - \epsilon < \gr_t(K) < 1$.  
This implies, in particular, that the spectrum of the growth rate is infinite.
\end{rmkk}

\begin{rmkk}
\label{rmk:RecentWork}
We take this opportunity to mention a few recent results about $b_i^*$
that appeared since we first started writing this paper; 
for precise statements see references.
\begin{enumerate}
\item In~\cite{IchiharaSaito}, given positive integers $g_M <i \leq  g_K$ and $n$,  
K Ichihara and T Saito constructed 
manifolds $M$ and knots $K \subset M$ 
so that $g(M) = g_M$, $g(E(K)) = g_K$, and $b_i^*(K) - b_{i-1}^*(K) \geq 2$
(see~\cite[Corolloar~2]{IchiharaSaito}; the notation there is different from ours);
their arguments can easily be applied to construct knots such that $b_i^*(K) - b_{i-1}^*(K) \geq n$ 
(informally, we may phrase this as an a{\it rbitrarily large gap}).  
\item In~\cite{Zupan} Zupan studies the
bridge indices of iterated torus knots showing, in particular, that there exist  
iterated torus knots realizing arbitrarily large gaps between $b_{i-1}^*$ and $b_i^*$ for 
any $i$ in the range where both indices are defined.
An easy argument shows that iterated torus knots are m-small; every knot $K$
considered by Zupan fulfils  $b_1^*(K) = 1$, and so has $\gr(K) = 0$
by Corollary~\ref{cor:bounds}.
\item In~\cite{BowmanTaylorZupan} Bowman, Taylor, and Zupan calculate the
bridge indices of generic iterated torus knots (see~\cite{BowmanTaylorZupan} 
for definitions).  They give conditions on the parameters that
imply that $b_g^* = p$, where here the knot considered is obtained by twisting the torus knot $T_{p,q}$, $p < q$. 
(We note that for twisted torus knot $g=2$).  Applying this to $K_{MSY}$ we
see that all but finitely many of these knots have $b_2^*=7$, improving on
our estimate $b_2^* \geq 4$.  We remark that in~\cite{BowmanTaylorZupan} 
linear lower bound on $b_1^*$ was also obtained, showing that many twisted torus knots
have a gap between $b_1^*$ and $b_2^*$; since $b_2^*$ can be made
arbitrarily large, this can be seen as a second gap.
\end{enumerate}
\end{rmkk}

\bigskip
\noindent
Before describing the structure and contents of this paper in more detail we introduce some necessary
concepts.  Let $\s$ be a Heegaard surface of a compact 3-manifold $M$, 
and $A$ an essential annulus properly embedded in $M$.
The annulus $A$ is called 
a {\it Haken annulus for $\s$} (Definition~\ref{dfn:haken annulus}) if it intersects $\s$ in a single simple closed curve that is essential in $A$.
For an integer $c \geq 0$, the manifold obtained by drilling $c$ curves simultaneously 
parallel to meridians of $K$ out of $E(K)$ is denoted by $E(K)^{(c)}$ (note that $E(K)^{(0)} = E(K)$).
The $c$ tori $\del E(K)^{(c)} \setminus \del E(K)$ are denoted by $T_1,\dots,T_c$.  There are $c$
annuli properly embedded disjointly in $E(K)^{(c)}$, denoted by $A_1,\dots,A_c$, so that one
component of $\del A_i$ is a meridian on $\del E(K)$ and the other is a longitude
of $T_i$ ($i=1,\dots,c$).  (We note that in general these annuli are not  
uniquely determined up to isotopy.)  Annuli with these properties are called 
{\it a complete system of Hopf annuli} 
 (Definition~\ref{dfn:hopf-annuli}).   
Let $\s$ be a Heegaard surface for $E(K)^{(c)}$.  
The Hopf annuli $A_1,\dots,A_c$ are called 
{\it a complete system of Hopf--Haken
Annuli for $\s$} 
(Definition~\ref{dfn:hopf-haken-annuli}) 
if $\s \cap A_i$ is a single simple closed curve that is essential in
$A_i$ ($i=1,\dots, c$).

Part~2 starts with Section~\ref{sec:haken annuli-definitions}
where we describe basic behavior of
Haken annuli under amalgamation.  
In Section~\ref{sec:tfae} we consider $(g', b)$ decomposition of $K$ 
(that is, $b$-bridge decomposition of $K$ with respect to a genus $g'$ Heegaard surface) 
and relate
it to existence of Hopf Haken Annuli.  Specifically, we prove
that $K$ admits a $(g(E(K))-c,c)$ decomposition if and only if $E(K)^{(c)}$ admits a
complete system of  Hopf Haken Annuli for some Heegaard surface of genus $g(E(K))$
(Theorem~\ref{thm:tfae}).  

In Section~\ref{sec:sft} we prove that given knots $K_1,\dots,K_n$ and
integers $c_1,\dots,c_n \geq 0$ with $\sum_{i=1}^n c_i = n-1$,
$E(K_1\#\cdots\#K_n)$ admits a system of $n-1$ essential tori 
$\mathcal{T}$ 
(called {\it swallow follow tori}) 
so that the components of 
$E(K_1\#\cdots\#K_n)$ cut open along $\mathcal{T}$ are homeomorphic to 
$E(K_1)^{(c_1)},\dots,E(K_n)^{(c_n)}$.  
By amalgamating Heegaard surfaces of
$E(K_1)^{(c_1)},\dots,E(K_n)^{(c_n)}$ along the tori of $\mathcal{T}$ 
we obtain a Heegaard surface for $E(K_1\#\cdots\#K_n)$; 
this implies the following special case of Corollary~\ref{cor:sft}:
$$g(E(K_1\#\cdots\#K_n)) \leq \sum_{i=1}^{n} g(E(K_i)^{(c_i)}) - (n-1)$$

In the final section of Part~2, Section~\ref{sec:upper-bound}, we combine
these facts to prove that for each $i$ we have: 
$$\gr_t(K) \leq 1- i/b_i^*$$ 
Thus we obtain the upper bound stated in Theorem~\ref{thm:main}.

\bigskip
\bigskip
\noindent To some degree, Part~3 complements Part~2.   We begin with
Section~\ref{sec:strongHH} that complements
Sections~\ref{sec:haken annuli-definitions} and \ref{sec:tfae}.  As mentioned
above, in Sections~\ref{sec:haken annuli-definitions} and \ref{sec:tfae} we
prove that $K$ admits a $(g(E(K))-c,c)$ decomposition if and only if $E(K)^{(c)}$
admits a
complete system of  Hopf Haken Annuli for some Heegaard surface of genus $g(E(K))$.
We are now ready to state the Strong Hopf
Haken Annulus Theorem,
which generalise the Hopf Haken Annulus Theorem (Theorem 6.3 of~\cite{kobayashi-rieck-m-small})
is one of the highlights of this work.  The 
proof is given in Section~\ref{sec:strongHH}. 
For the definition of a Heegaard splitting of $(N;F_1,F_2)$
(where $N$ is a manifold and $F_1$, $F_2$ are partitions of
some of the components of $\partial N$), see Section~\ref{sec:background} .

\begin{thm}[Strong Hopf-Haken Annulus Theorem]
\label{thm:strongHH}
For $i=1,\dots,n$, let $M_i$ be a compact connected orinetable 3-manifold and $K_i \subset M_i$ a knot.
Suppose $E(K_i) \not\cong T^{2} \times I$, $E(K_i)$ is irreducible,
and $\del N(K_i)$ is incompressible in $E(K_i)$. 
Let $F_1$, $F_2$ be a partition of some of the components of  
$\del M$, where $M = \#_{i=1}^{n} M_i$.  Let $c \geq 0$ be an integer.  
Then one of the following holds:

\begin{enumerate}
\item  There exist a minimal genus Heegaard surface for
$(E(\#_{i=1}^{n}K_{i})^{(c)}; F_1, F_2)$ admitting a 
complete system of Hopf--Haken annuli.
\item For some $1 \leq i \leq n$, $E(K_i)$ admits an essential meridional
  surface $S$ with $\chi(S) \geq  \genbd[g(E(\#_{i=1}^{n}K_{i})^{(c)}; F_1, F_2)]$. 
\end{enumerate}

\end{thm}  
%
%For clarification of the statement of the theorem above we note the following.
%Recall that $T_1, \dots , T_c$ are the tori corresponding to 
%the boundary of the regular neighborhood of the $c$ boundary  
%parallel curves drilled.  Then we have: 
%$$\del M = \del E(\#_{i=1}^{n}K_{i})^{(c)} \setminus \big( (\cup_{i=1}^c T_i) \cup \del N(\#_{i=1}^n K_i)\big)$$

One curious consequence of Theorem~\ref{thm:strongHH} (which is proved in
Section~\ref{sec:strongHH}) is the following, where $b_g^*$ is as in Notation~\ref{notation:bridge indices}:

\begin{cor}
\label{cor:GenusOfXc}
Let $K \subset S^{3}$ be a connected sum of m-small knots.  Then for $c \ge b_{g}^{*}$,
$$g(E(K)^{(c)}) = c$$
\end{cor}

Section~\ref{sec:minimal-genus-sft} complements Section~\ref{sec:sft}.  Recall
that in Section~\ref{sec:sft} we used swallow follow tori to show that
given {\it any} collection
of integers $c_1,\dots,c_n \geq 0$ whose sum is $n-1$ we have that
$g(E(K_1\#\cdots\#K_n)) \leq \sum_{i=1}^{n} g(E(K_i)^{(c_i)}) - (n-1)$.
In Section~\ref{sec:minimal-genus-sft} we prove that if $K_i$ is m-small for each
$i$, then any Heegaard splitting for $E(K_1\#\cdots\#K_n)$ admits an iterated weak reduction 
to $n-1$ swallow follow tori.  
%This gives a converse to the results of Section~\ref{sec:sft}.  
%We point out an important distinction between Sections~\ref{sec:sft} and~\ref{sec:minimal-genus-sft}:
%in Section~\ref{sec:sft} we prove that given {\it any} collection
%of integers $c_1,\dots,c_n \geq 0$ whose sum is $n-1$ there
%exist swallow follow tori $T_1,\dots,T_{n-1}$  that decompose
%$E(K_1\#\dots\#K_n)$ as $E(K_1)^{(c_1)},\cdots,E(K_n)^{(c_n)}$.  
%By contrast, in Section~\ref{sec:minimal-genus-sft} we prove 
This implies that any minimal genus 
Heegaard splitting admits an iterated weak reduction
to {\it some} $n-1$ swallow follow tori that decompose $E(K_1\#\cdots\# K_n)$ as
$E(K_1)^{(c_1)},\dots,E(K_n)^{(c_n)}$, giving {\it some} integers 
$c_1,\dots,c_n \geq 0$ whose sum is $n-1$.  
The integers $c_1,\dots,c_n$ are very
special (see Example~\ref{ex:sft2}).

In Section~\ref{sec:growth-rate}, which complements Section~\ref{sec:upper-bound}, 
we combine these results to give a
lower bound on the growth rate of the tunnel number of m-small knots.  Given
$K$, we ``expect'' that $g(E(K)^{(c)}) = g(E(K)) + c$; we define the function
$\fff$ that measures to what extent $g(E(K)^{(c)})$ fails to behave ``as expected'':
$$\fff(c) = g(E(K)) + c - g(E(K)^{(c)})$$
For any knot $K$ and any integer $c \geq 0$, we show that $\fff$ fulfills:
$$\fff(0)=0 \mbox{   and  } \fff(c)\le\fff(c+1)\le\fff(c)+1$$

We study $\fff$ for m-small knots, calculating it exactly in terms of the
bridge indices of $K$ (Proposition~\ref{pro:fk}).  
In particular, for m-small knots $\fff$ is bounded.  
In fact, for large enough $c$ Proposition~\ref{pro:fk} implies: 
$$\fff(c) = g(E(K)) - g(M)$$ 
We do not know much about the behavior of $\fff$ in general;
for example, we do not know if there exists a knot for which $\fff$ is unbounded
(see Question~\ref{que:bounded}).

We express the growth rate of tunnel number of $m$-small knots 
in terms of $\fff$ by showing (Corollary~\ref{cor:sft-strong4}) that:
$$\frac{g(E(nK)) - ng(E(K)) + n-1}{n-1} = 1 - \frac{\max\{\sum_{i=1}^n
  \fff(c_i)\}}{n-1}$$ 
where the maximum is taken over all collections of integers $c_1,\dots,c_n \geq 0$ 
whose sum is $n-1$.  
The growth rate is then the limit superior of this
sequence.  
We combine this interpretation of the growth rate with the calculation of
$\fff$ to obtain the exact calculation of the growth rate of m-small knots 
stated in Theorem~\ref{thm:main}.
 
 \bigskip
 
\noindent {\bf Acknowledgements.}  We thank Mark Arnold and Jennifer Schultens
for helpful correspondence.

\section{Preliminaries}
\label{sec:background}
  
By {\it manifold} we mean a smooth 3 dimensional manifold.  All
manifolds considered are assumed to be connected orientable and compact.  
We assume the reader is familiar with the basic terms of 3-manifold topology (see for example \cite{hempel} or \cite{jaco}).
Thus we assume the reader is familiar with terms such as compression, boundary compression,
boundary parallel, and essential surface.

We use the notation $\partial$,  $\mbox{cl}$, and $\mbox{int}$ denote boundary, 
closure, and interior, respectively.  For a submanifold 
$H$ of a manifold $M$, $N(H,M)$ denotes a closed regular neighborhood of $H$ in $M$.  When $M$
is understood from context we often abbreviate $N(H,M)$ to $N(H)$.

By a {\it knot} $K$ in a 3-manifold $M$ we mean a smooth embedding of $S^1$ into $M$, taken
up to ambient isotopy.
The {\it exterior} of $K$, $E(K)$, is $\mbox{cl}(M \setminus N(K))$.  The slope on the torus
$\partial E(K) \setminus \partial M = \partial N(K)$ that bounds a disk in $N(K)$
is called the {\it meridian} of $K$.
A knot $K$ is called {\it m-small} if there is no essential meridional surface
in $E(K)$, that is, there is no essential surface $S \subset E(K)$ with non
empty boundary so that $\del S$ consists of meridians of $K$.

We assume the reader is familiar with the basic terms regarding Heegaard splittings, such as 
handlebody, compression body, meridian disk, etc.  
Recall that a compression body $C$ is a connected 3-manifold obtained from $F \times [0,1]$
(where here $F$ is a possibly empty disjoint union of closed surfaces)
and a (possibly empty) collection of 3-balls by attaching 1-handles to
$F \times \{1\}$ and the boundary of the balls.   
Following standard conventions, we refer to $F \times \{0\}$ as
$\del_- C$ and $\partial C \setminus \del_{-} C$ as $\del_+  C$.
We use the notation $C_1 \cup_\s C_2$ 
for the Heegaard splitting given by the compression bodies $C_1$ and $C_2$.
The basic concepts of reductions of a Heegaard splitting are summarized here:

\begin{dfns}
\label{dfn:reducibility of heegaard splittings}

\begin{enumerate}
    \item  A Heegaard splitting $C_1 \cup_\s C_2$ is called
    \em stabilized \em if there exist meridian disks $D_1 \subset
    C_1$ and $D_2 \subset C_2$ such that $\del D_1$ intersects $\del D_2$
    transversely (as submanifolds of $\s$) in one point.
    Otherwise, the Heegaard splitting is called \em non-stabilized. \em
    \item A Heegaard splitting $C_1 \cup_\s C_2$ is called
    \em reducible \em if there exist meridian disks $D_1 \subset
    C_1$ and $D_2 \subset C_2$ such that $\del D_1 = \del D_2$.
    Otherwise, the Heegaard splitting is called \em irreducible. \em
    \item A Heegaard splitting $C_1 \cup_\s C_2$ is called
    \em weakly reducible \em if there exist meridian disks $D_1 \subset
    C_1$ and $D_2 \subset C_2$ such that $\del D_1 \cap \del D_2 =
    \emptyset$.  Otherwise the splitting is called \em strongly
    irreducible.\em
    \item A Heegaard splitting $C_1 \cup_\s C_2$ is called
    \em trivial \em if $C_1$ or $C_2$ is a trivial compression body,
    that is, a compression body with no 1-handles.
    Otherwise the Heegaard splitting is called \em non-trivial.\em
\end{enumerate}
\end{dfns}

Let $C_1 \cup_\s C_2$ be a weakly reducible Heegaard splitting of a manifold $M$.  Let $\Delta_i \subset C_i$
be a non empty set of disjoint meridian disks so that $\Delta_1 \cap \Delta_2 = \emptyset$.	
By {\it waek reduction} along $\Delta_1 \cup \Delta_2$ we mean the (possibly disconnected) surface obtained by first compressing
$\s$  along $\Delta_1 \cup \Delta_2$, and then removing any component that is contained in $C_1$ or $C_2$.
Casson and Gordon~\cite{casson-gordon} showed that if an irreducible Heegaard splitting is weakly reducible,
then an appropriately chosen weak reduction yields a (possibly disconnected) essential surface, say $F$.

With $F$ as in the previous paragraph, let $M_1,\dots,M_k$ be the components of $M$ cut open along
$F$.  It is well known that $\s$ induces a Heegaard surface on each $M_i$, say 
$\s_i$.  We say that $\s$ is obtained by {\it amalgamating} $\s_1,\dots,\s_k$.
This is a special case of amalgamation; the general definition will be given below
as the converse of iterated weak reduction.
The genus after amalgamation is given in the following lemma;
see Remark~2.7 of~\cite{schultens-FXS1} for the case $m=1$
(we leave the proof of the general case to the reader):

\begin{lem}
\label{lem:genus after amalgamation}
Let $C_1 \cup_{\s} C_2$ be a weakly reducible Heegaard splitting
and suppose that after weak reduction we obtain $F$ (as above).
Suppose that $M$ cut open along $F$ consists of two components, and denote the induced Heegaard splittings by
$C_1^{(1)} \cup_{\s_1} C_2^{(1)}$ and $C_1^{(2)} \cup_{\s_2} C_2^{(2)}$.  
Let $F_1,\dots,F_m$ be the components of $F$. 
Then
$$g(\s) = g(\s_1) + g(\s_2) - \sum_{i=1}^m g(F_i) + (m-1)$$
In particular, if $F$ is connected then 
$g(\s) = g(\s_1) + g(\s_2) -  g(F)$.
\end{lem}

It is distinctly possible that not all the Heegaard splittings induced by weak reduction are strongly
irreducible.  When that happens we may weakly reduce some (possibly all) of the induced Heegaard
splitting, and repeat this process.  We refer to this as {\it repeated} or {\it iterated}
weak reduction.  The converse is called amalgamation.
Scharlemann and Thompson~\cite{scharl-abby} proved that any Heegaard splitting
admits a repeated  weak reduction so that the induced Heegaard splittings are all
strongly irreducible; we refer to this as {\it untelescoping}.

Let $N$ be a manifold and $\{F_{1},F_{2}\}$ a partition of some components of $\del N$.
Note that we do {\it not} require every component of $\partial N$ to be in $F_1$ or $F_2$.
We say that $C_1 \cup_\s C_2$ is a Heegaard splitting of $(N;F_1,F_2)$
if $F_1 \subset \del_- C_1$ and $F_2 \subset \del_- C_2$.
We extend the terminology of Heegaard splittings to this context, so, for
example, $g(N;F_1,F_2)$ is the genus of a minimal genus Heegaard splitting
 of $(N;F_1,F_2)$.

\medskip

The following proposition allows us, in some cases, to consider weak reduction instead of 
iterated weak reduction.  The proof is simple and left to the reader.

\begin{pro}
\label{pro:untel to a conn sep surfce implies weak reduction}
Let $F$ be a component of the surface obtained by repeated weak reduction of
$C_1 \cup_{\s_1} C_2$.  If $F$ is separating, then some weak reduction of $C_1 \cup_{\s_1} C_2$
yields exactly $F$.

\end{pro}

\section{Relative Heegaard Surfaces}
\label{sec*RelativeSplittings}

In this section we study relative Heegaard surfaces.  The results of this section will be used
in Section~\ref{sec:strongHH} and the reader may postpone reading it until that section.  
Let $b \geq 1$ be an integer and $T$ be a torus.  For $1 \leq i \leq 2b$, 
let $A_{i} \subset T$ be an annulus.
We say that $\{A_1,\dots,A_{2b}\}$ gives a {\it decomposition 
of $T$ into annuli} (or simply a {\it decomposition of $T$}) if the following two conditions hold:
\begin{enumerate}
\item $\cup_{i=1}^{2b} A_i = T$, and
\item 
	\begin{enumerate}
	\item If $b >1$, then $A_i \cap A_j = \emptyset$ whenever $i \ne j$ are non consecutive integers (modulo $2b$),
	and $A_i \cap A_{i+1} = \partial A_{i} \cap \partial A_{i+1}$ is a single simple closed curve. 
	\item If $b = 1$, then $A_1 \cap A_2 = \partial A_1  = \partial A_2$.
	\end{enumerate}
\end{enumerate}
We begin by defining a relative Heegaard surface; note that the definition can be made more general
by considering an arbitrary collection of boundary components (below we only consider a single torus)
and a decomposition into arbitrary subsurfaces (below we only consider annuli); however the definition
below suffices for our purposes:

\begin{dfn}[relative Heegaard surface]
\label{def:RalitiveSplitting}
Let $M$ be a compact connected orientable 3-manifold and $T$ a torus component of $\partial M$.
Let $\{A_1,\dots,A_{2b}\}$ be a decomposition of $T$ into annuli.
A compact surface $S \subset M$ is called a \em Heegaard surface for $M$ relative to $\{A_1,\dots,A_{2b}\}$ \em
(or simply a \em relative Heegaard surface, \em when no confusion may arise) if the following conditions hold:
\begin{enumerate}
\item $\partial S = \cup_{i=1}^{2b} \partial A_i$,
\item $M$ cut open along $S$ consists of two components (say $C_1$ and $C_2$), 
\item For $i=1,2$, $C_{i}$ admits a set of compressing disks $\Delta_{i}$ 
with $\partial \Delta_i \subset S$, so that $C_{i}$ compressed along $\Delta_{i}$ consists of:
	\begin{enumerate}
	\item exactly $b$ solid tori, each containing exactly one $A_{i}$ as a longitudinal annulus;
	\item a (possibly empty) collection of collar neighborhoods of components of $\partial M \setminus T$;
	\item a (possibly empty) collection of balls.
	\end{enumerate}
\end{enumerate}
The genus of a minimal genus relative Heegaard surface is called the \em relative genus.\em
\end{dfn}

For an integer $c \geq 1$, let $Q^{(c)}$ be (annulus with $c$ holes) $\times S^{1}$.
(To avoid confusion we remark that $Q^{(c)}$ can be described as (sphere with $c+2$ 
holes) $\times S^{1}$, but in the context of this paper an annulus is more natural.) 
Note that $Q^{(c)}$ admits a unique Seifert fiberation.
Our goal is to calculate the genus of $Q^{(c)}$ relative to a given decomposition of a 
component of $\partial Q^{(c)}$ into annuli.  We say that
slopes $\beta$ and $\gamma$ of a torus are 
\em complimentary \em if they are represented by simple closed curves that intersect each other 
transversely once.  

\begin{pro}
\label{prop:RelativeGenus}
Let $\{A_1,\dots,A_{2b}\}$ be a decomposition of a component of $\partial Q^{(c)}$ (say $T$) into annuli, and denote
that the slope defined by these annuli by $\beta$.  Denote the slope defined by the Seifert fibers on $T$
by $\gamma$.  Then we have:
\begin{itemize}
\item When $\beta$ and $\gamma$ are complimentary slopes,  the genus of $Q^{(c)}$ relative to $\{A_1,\dots,A_{2b}\}$ is $c$.
\item When $\beta$ and $\gamma$ are not complimentary slopes,  the genus of $Q^{(c)}$ relative to $\{A_1,\dots,A_{2b}\}$ is $c+1$.
\end{itemize}
\end{pro}

We immediately obtain:

\begin{cor}
\label{cor:RelHSforQc}
The surfaces in Figure~\ref{fig:RelHS} are minimal genus Heegaard splitting for $Q^{(c)}$
relative to $\{A_{1},\dots,A_{2b}\}$; the left figure is complimentary slopes and the right figure is
non-complimentary slopes.
\end{cor}

\begin{figure}[h!]
\includegraphics[height=1.8in]{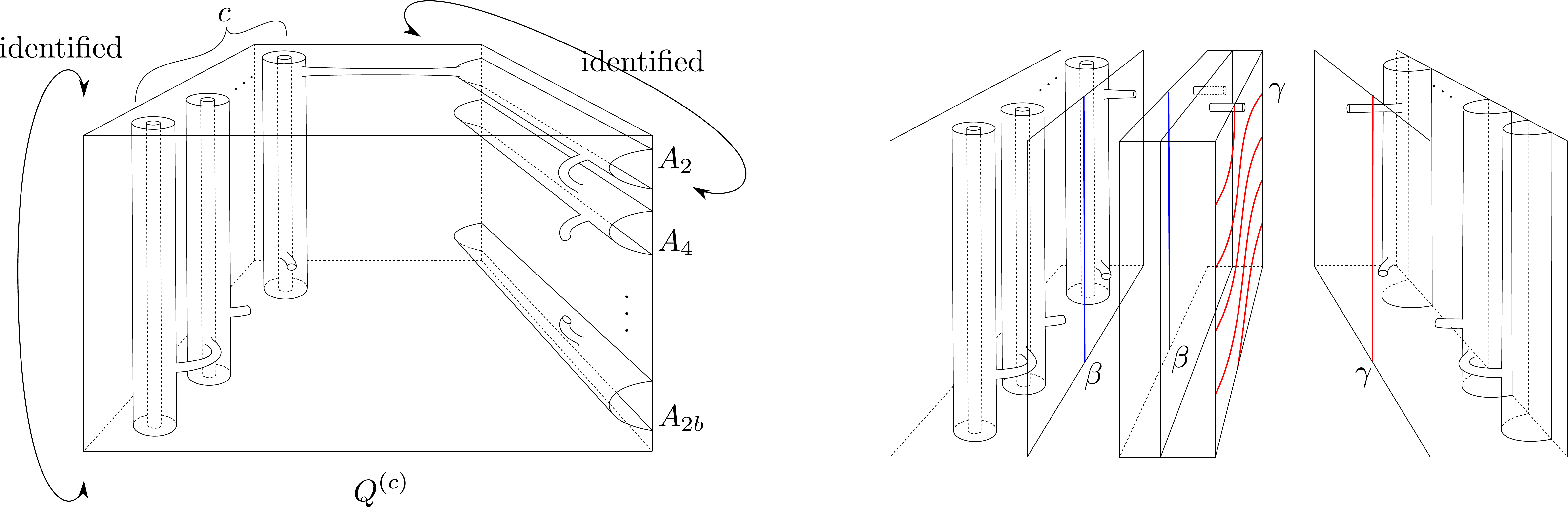}
\caption{Relative Heegaard surfaces}
\label{fig:RelHS}
\end{figure}

We postpone the proof of Proposition~\ref{prop:RelativeGenus} to the end of this section, as
it will be an application of the next proposition which is of independent interest.
We fix the following notation: glue $Q^{(b)}$ to $Q^{(c)}$ along a single boundary component and denote the 
slope of the Seifert fiber of $Q^{(b)}$ on the torus  $Q^{(b)} \cap Q^{(c)}$ by $\beta$
and the slope of the Seifert fiber of $Q^{(c)}$ by $\gamma$.
The manifold obtained is denoted \qbc.

\begin{pro}
\label{pro:GenusOfQbc}
The genus of \qbc\ fulfils:
\begin{itemize}
\item If $\beta$ and $\gamma$ are complimentary slopes, then $g\big(\qbc\big) = b+c$.
\item If $\beta$ and $\gamma$ are not complimentary slopes, then $g\big(\qbc\big) = b+c+1$.
\end{itemize}
\end{pro}

We immediately obtain:

\begin{cor}
\label{cor:HSforQbc}
The surfaces in Figure~\ref{fig:HSforQbc} are minimal genus Heegaard splitting for $\qbc$.
\end{cor}

\begin{figure}[h!]
\includegraphics[height=1.6in]{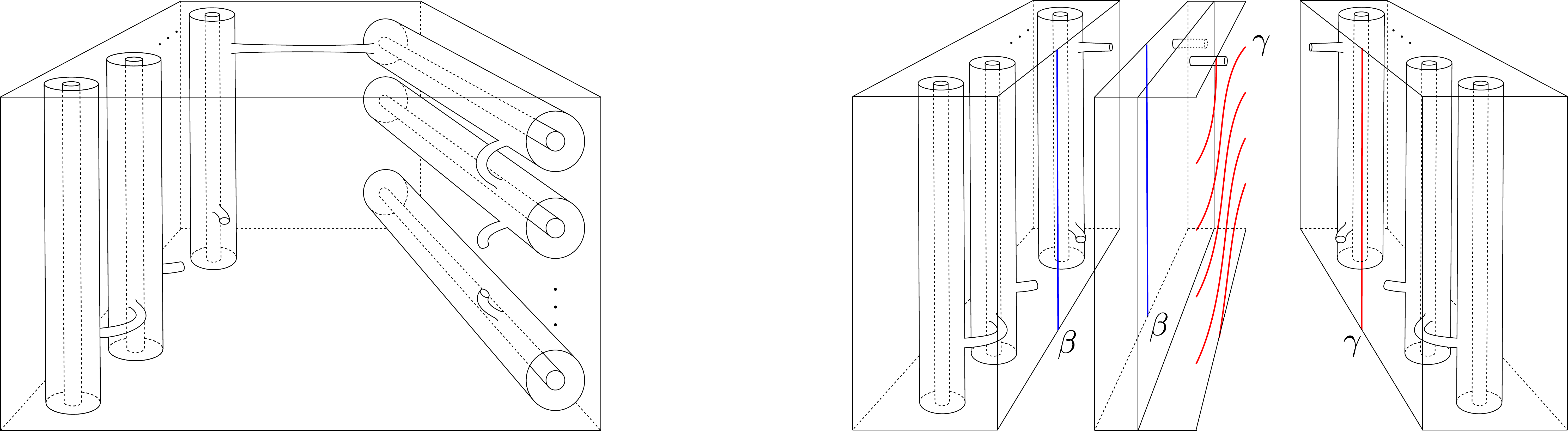}
\caption{Heegaard surfaces for \qbc}
\label{fig:HSforQbc}
\end{figure}

A surface in a Seifert fibered space is called {\it vertical} if it is everywhere tangent
to the fibers and {\it horizontal} if it is everywhere transverse to the fibers.   It is well
known that given an essential surface in a Seifert fibered space we may assume it is
vertical or horizontal;  see for example,~\cite{jaco}.

\begin{proof}[Proof of Proposition~\ref{pro:GenusOfQbc}]
The surfaces in Figure~\ref{fig:HSforQbc} are Heegaard surfaces
for \qbc, showing the following, which we record here 
for future reference:

\begin{rmkk}
\label{rmk:graph}
When $\beta$ and $\gamma$ are complimentary, $g\big(\qbc\big) \leq b+c$. 
When $\beta$ and $\gamma$ are not complimentary, $g\big(\qbc\big) \leq b+c+1$.
\end{rmkk}

Hence we only need to show that
when $\beta$ and $\gamma$ are complimentary, $g\big(\qbc\big) \geq b+c$ and 
when $\beta$ and $\gamma$ are not complimentary, $g\big(\qbc\big) \geq b+c+1$.

If $\beta = \gamma$ then \qbc\ is a $b+c$ times punctured annulus cross $S^1$ and the result
was proved by Schultens in~\cite{schultens-FXS1}.  For the remainder of the proof we assume 
that $\beta \ne \gamma$.
Then \qbc\ is a graph manifold whose underlying graph consists of two vertices connected by a single edge.
We apply Theorem~1.1 of Schultens~\cite{schultens-graph} 
and refer the reader to that paper for notation and details.
To conform to its noation, following~\cite{schultens-graph}, we decompose \qbc\ 
along two parallel copies of $Q^{(b)} \cap Q^{(c)}$
as $\qbc = Q_{b} \cup M_{e} \cup Q_{c}$.
%(see Figure~\ref{fig:QbcAsQbMeQc}).  
$Q_{b}$ and $Q_{c}$ are called the {\it vertex manifolds} and $M_{e}$ is the {\it edge manifold}.
Note that $Q_{b} \cong Q^{(b)}$, $M_{e} \cong T^{2} \times [0,1]$, and  $Q_{c} \cong Q^{(c)}$.

%\begin{figure}[h!]
%\includegraphics{filename4}
%\caption{Decomposition of \qbc\ as $Q_b \cup M_e \cup Q_{c}$}
%\label{fig:QbcAsQbMeQc}
%\end{figure}

Let $S$ be a minimal genus Heegaard splitting for \qbc.  In the following claim we analyze 
completely what happens when $g(S) = 2$ or when $S$ is strongly irreducible:
	
\begin{clm}
\label{clm:SisSIinQbc}  The following three conditions are equivalent:
\begin{enumerate}
\item $S$ is strongly irreducible. 
\item The following conditions hold:
\begin{itemize}
\item $\beta$ and $\gamma$ are complimentary.
\item $g(S) = 2$.
\item $b=c=1$.
\end{itemize}
\item $g(S) = 2$.
\end{enumerate}
\end{clm}

\begin{proof}[Proof of Claim~\ref{clm:SisSIinQbc}] (1) implies (2). 
Suppose that $S$ is strongly irreducible.  
By~\cite{schultens-graph} we may assume that $S$ is standard.  In particular, 
$S \cap Q_b$ (respectively  $S \cap Q_c$) is either horizontal, pseudohorizontal, vertical, or pseudovertical.
However, the first two cases are impossible as they require $S$ to meet every boundary component
of $Q_b$ (respectively $Q_c$).  Hence $S \cap Q_b$ and  $S \cap Q_c$
consist of vertical or pseudovertical components.
In particular, the intersection of $S$ with the torus $Q_b \cap M_e$  (respectively $Q_c \cap M_e$)
is a Seifert fiber of $Q_b$ (respectively $Q_c$).

Assume first that $S \cap M_e$ is as in Case~(1) of~\cite[Theorem~1.1]{schultens-graph},
that is,  $S \cap M_e$ is obtained from a collection incompressible annuli, say $\mathcal{A}$,
by tubing along at most one boundary parallel arc (in~\cite{schultens-graph}, tubings are referred to as {\it 1-surgary}).  
Suppose that $\mathcal{A}$ consists of boundary parallel annuli.  Since the tubing is performed, if at all,
along a boundary parallel arc, we see that no component of $S \cap M_e$ connects the components of
$\partial M_e$.  This contradicts the fact that $S$ is connected and must meet both $Q_b$ and $Q_c$.
Hence some component of $\mathcal{A}$ meets both components of $\partial M_e$, showing that $\beta = \gamma$, 
contradicting our assumption.

Hence Case~(2) of~\cite[Theorem~1.1]{schultens-graph} holds, and
$S \cap M_e$ consists of a single
component that is obtained by tubing together two boundary parallel annuli, one at each boundary
component of $M_e$; moreover, \cite[Theorem~1.1]{schultens-graph} shows that 
these annuli define complementary slopes. See the left side of Figure~\ref{fig:HSinMe}. 
As argued above, the slopes 
defined by these annuli are $\beta$ and $\gamma$.  This gives the first condition of~(2).

\newpage

\begin{figure}[h!]
\includegraphics[height=1.5in]{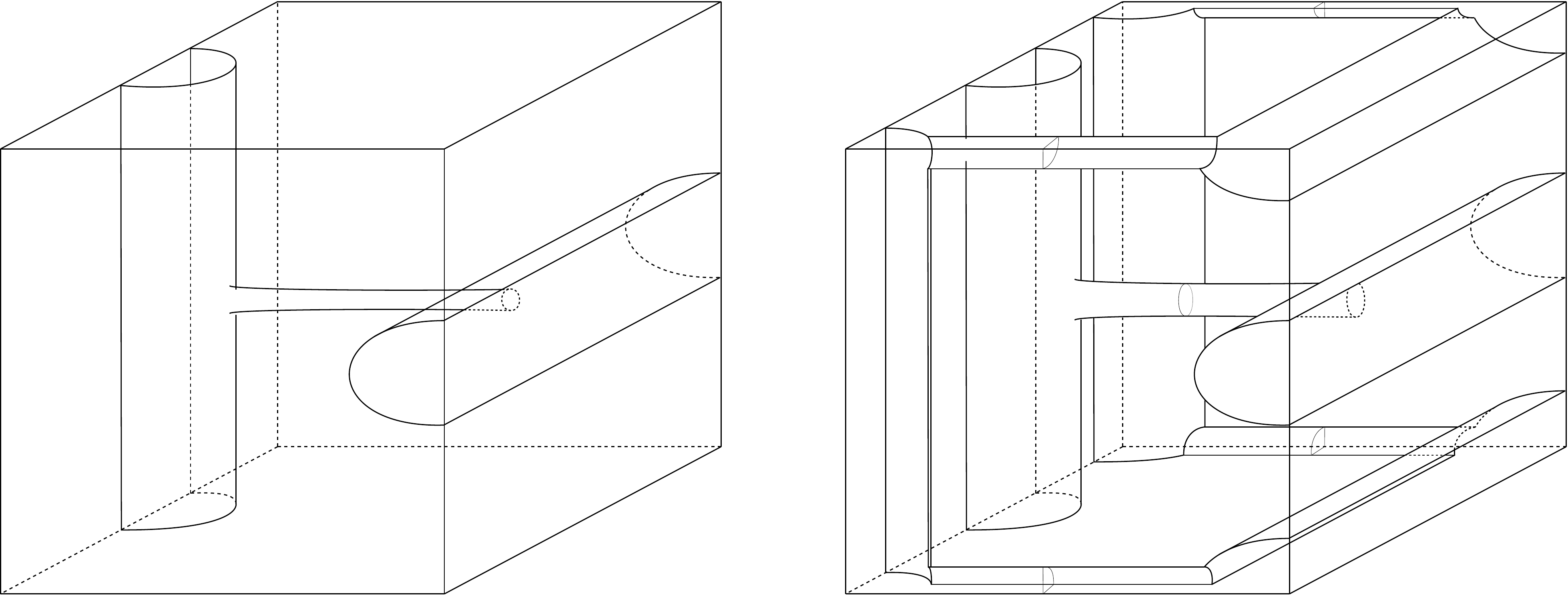}
\caption{Heegaard surface in $M_{e}$}
\label{fig:HSinMe}
\end{figure}

On the right side of Figure~\ref{fig:HSinMe} we see two surfaces.
One is $S \cap M_e$, and in its center we marked the boundary of the obvious compressing disk.
It is easy to see that the other surfce is isotopic to $S \cap M_e$.
On it we marked the boundary of four disks, each shaped like $90^o$ sector.
After gluing of opposite sides of the cube to obtained $M_e$, these sectors form a compressing 
disk on the opposite side of the obvious disk.
This demonstrates that $S \cap M_e$ compresses into both sides.
If $S \cap Q_b$ is pseudovertical then it compresses, and together with one of the compressing
disks for $S \cap M_e$ we obtain a weak reduction, contradicting our assumption.
Hence $S \cap Q_b$ consists of annuli; similarly,  $S \cap Q_c$ consists of annuli.  
Hence $\chi(S) = \chi(S \cap M_{e}) = -2$.  The second condition of~(2) follows.

Since $g(S) = 2$, $\partial \qbc$ consists of at most four tori.  On the other hand, $\partial \qbc$
consists of $b+c+2$ tori, for $b,c \geq 1$.  Hence $b=c=1$, fulfilling
the third and final condition of~(2).
This completes the proof that~(1) implies~(2).

It is trivial that~(2) implies~(3).

To see that~(3) implies~(1), assume that $S$ weakly reduces.  Since $S$ is a minimal genus
Heegaard surface and $g(S) = 2$, an appropriate weak reduction yields an essential sphere,
contradicting the fact that \qbc\ is irreducible.

This completes the proof of Claim~\ref{clm:SisSIinQbc}.
\end{proof}

If $S$ is strongly irreducible, Proposition~\ref{pro:GenusOfQbc} follows from Claim~\ref{clm:SisSIinQbc}.
For the reminder of the proof we assume as we may that $S$ weakly reduces to a (possibly 
disconnected) essential surface, say $F$.  By the construction of \qbc\ we see that every component of $F$
separates; hence by Proposition~\ref{pro:untel to a conn sep surfce implies weak reduction} 
we may assume that $F$ is connected.  Recall that we assumed that $\beta \neq \gamma$.
This clearly implies that we may suppose that (after isotopy if necessary) $F$ is disjoint from the torus $Q^{(b)} \cap Q^{(c)}$;
without loss of generality we assume that $F \subset Q^{(b)}$.  

We induct on $b+c$.

\bigskip

\noindent {\bf Base case: $b+c = 2$.}  Note that in the base case $b = c = 1$. 
It is easy to see that the only connected essential surface in  $Q^{(1,1)}_{\beta,\gamma}$ is the torus 
$Q^{(b)} \cap Q^{(c)}$.  Hence $F$ is isotopic to this surface and the weak reduction 
induces Heegaard splittings $\s_b$ and $\s_c$ on $Q^{(b)}$ and $Q^{(c)}$, respectively; 
note that both $Q^{(b)}$ and $Q^{(c)}$ are homeomorphic to $Q^{(1)}$.
By Schultens~\cite{schultens-FXS1}, $g(Q^{(1)}) = 2$.  
by Lemma~\ref{lem:genus after amalgamation} amalgamation gives:
$$g(Q^{(1,1)}_{\beta,\gamma}) = g(S) = g(\s_b) + g(\s_c) - g(F)
\ge g(Q^{(1)}) + g(Q^{(1)}) - g(F) = 2+2-1=3$$
By Remark~\ref{rmk:graph}, if $\beta$ and $\gamma$ are complimentary slopes
then $g(Q^{(1,1)}_{\beta,\gamma}) \leq 2$; hence 
$\beta$ and $\gamma$ are not complimentary slopes and
together with Remark~\ref{rmk:graph}
the proposition follows in this case.

\bigskip

\noindent {\bf Inductive case: $b+c > 2$.}  Assume, by induction, that 
the proposition holds for any integers $b', c' > 0$, with 
$b' + c' < b+c$.  

\bigskip
\noindent
{\bf Case One: $F$ is isotopic to  $Q^{(b)} \cap Q^{(c)}$.  }
Then weak reduction induces Heegaard 
splittings on $Q^{(b)}$ and  $Q^{(c)}$.
Similar to the argument above (using that
$g(Q^{(b)}) = b+1$ and $g(Q^{(c)}) = c+1$ by~\cite{schultens-FXS1}) we have,
$$g(\qbc) \ge g(Q^{(b)}) + g(Q^{(c)}) - g(F)  = b+c+1$$
As in the base case it follows from Remark~\ref{rmk:graph} that $\beta$ and $\gamma$ are not 
complimentary slopes.  Together with Remark~\ref{rmk:graph}, the proposition follows in this case.

\bigskip
\noindent
{\bf Case Two: $F$ is not isotopic to  $Q^{(b)} \cap Q^{(c)}$.  }
Then $F$ is essential in $Q^{(b)}$ and is therefore isotopic
to a vertical or horizontal surface. % (see, for example,~\cite[VI.34]{jaco}).
Since $F$ is closed and $\partial Q^{(b)} \neq \emptyset$, we have that $F$
cannot be horizontal.  We conclude that $F$ is
a vertical torus and decomposes $Q^{(b)}$ as $Q^{(b')}$
(for some $b' < b$) and a disk with $b - b' +1$ holes cross $S^1$.
By induction, the genus of $Q^{(b',c)}_{\beta,\gamma}$ fulfills the
conclusion of Proposition~\ref{pro:GenusOfQbc}; 
by~\cite{schultens-FXS1}, the genus of disk with $b-b'+1$ holes cross $S^1$
is $b-b'+1$; similar to the argument above we get
$$ g(\qbc) \ge g(Q^{(b',c)}_{\beta,\gamma}) + (b-b'+1) - 1 = g(Q^{(b',c)}_{\beta,\gamma}) + b-b'$$

Together with Remark~\ref{rmk:graph}, this completes the proof 
of Proposition~\ref{pro:GenusOfQbc}.
\end{proof}

We are now ready to prove Proposition~\ref{prop:RelativeGenus}:

\begin{proof}[Proof of  Proposition~\ref{prop:RelativeGenus}]
The surfaces in Figure~\ref{fig:RelHS} are relative Heegaard surfaces
realizing the values given in Proposition~\ref{prop:RelativeGenus}.
To complete the proof we only need to show that these surfaces realize
the minimal relative genus.

Let $\Sigma$ be a minimal genus Heegaard surface for $Q^{(c)}$ relative to 
$\{A_{1},\dots,A_{2b}\}$.  By tubing $\partial \Sigma$ along
the annuli $A_{2i}$ and drilling a curve parallel to the core
of $A_{2i}$ ($i=1,\dots,b$; recall Figure~\ref{fig:RelHS})
we obtain a Heegaard surface for \qbc\ of genus $g(S) + b$.  
Thus $g(\s) \geq g(\qbc) - b$.  By Proposition~\ref{pro:GenusOfQbc},
when $\beta$ and $\gamma$ are complimentary $g(\qbc) = b+c$
and when $\beta$ and $\gamma$ are not complimentary $g(\qbc) = b+c+1$.
Thus we see that $g(\s) \geq c$ (when the $\beta$ and $\gamma$ are complimentary)
and $g(\s) \geq c+1$ (otherwise).

This completes the proof of Proposition~\ref{prop:RelativeGenus}.
\end{proof}

\part{An upper bound on the growth rate of the tunnel number of knots} 
\label{part:upper-bound}

\section{Haken Annuli}
\label{sec:haken annuli-definitions}

A primary tool in our study are Haken annuli.
Haken annuli were first defined 
in~\cite{kobayashi-rieck-m-small}, 
where only a single annulus was considered.  We generalize the definition to a collection of annuli below.
Note the similarity between a Haken annulus and a Haken sphere  or Haken disk
(by a {\it Haken sphere} we mean a sphere that meets a
Heegaard surface in a single  \scc\ that is essential  
in the Heegaard surface, see \cite{haken} or
\cite[Chapter 2]{jaco}, and by a {\it Haken disk} we mean a disk that meets a
Heegaard surface in a single   \scc\ that is essential  
in the Heegaard surface \cite{casson-gordon}). 

\begin{dfn}
\label{dfn:haken annulus}

Let $C_1 \cup_\s C_2$ be a Heegaard splitting of a manifold $M$.
A collection of essential annuli $\mathcal{A} \subset M$ are called \em Haken annuli \em
for $C_1 \cup_\s C_2$ (or simply \em Haken annuli, \em when no confusion may arise)
if for every annulus $A \in \mathcal{A}$ we have that $A \cap \s$ consists of a single
\scc\ that is essential in $A$.

\end{dfn}

\begin{rmkk}
\label{rmk:HakenAannuliForDC}
For an integer $n \ge 2$, let $D(n)$ be a (disk with $n$ holes)$\times S^1$ 
and denote the components of $\del D(n)$ by $T_0, T_1, \dots , T_n$. 
By the construction of minimal genus Heegaard splittings given in the 
proof of Proposition~2.14 of \cite{kobayashi-rieck-m-small}, 
we see that for each positive integers $p$ with $1 \le p \le n$
there is a genus $n$ Heegaard surface of
$(D(n); \cup_{i=0}^{p-1} T_i, \cup_{i=p}^n T_i)$ which admits a collection $\{A_1,\dots,A_p\}$
of Haken annuli connecting $T_i$ to $T_n$ ($i=0,\dots,p-1$). 
By Schultens \cite{schultens-FXS1}, we see that this is a minimal genus Heegaard splitting of $D(n)$. 
See Figure~\ref{fig:HSofDn}.
\begin{figure}[h!]
\includegraphics[height=1.5in]{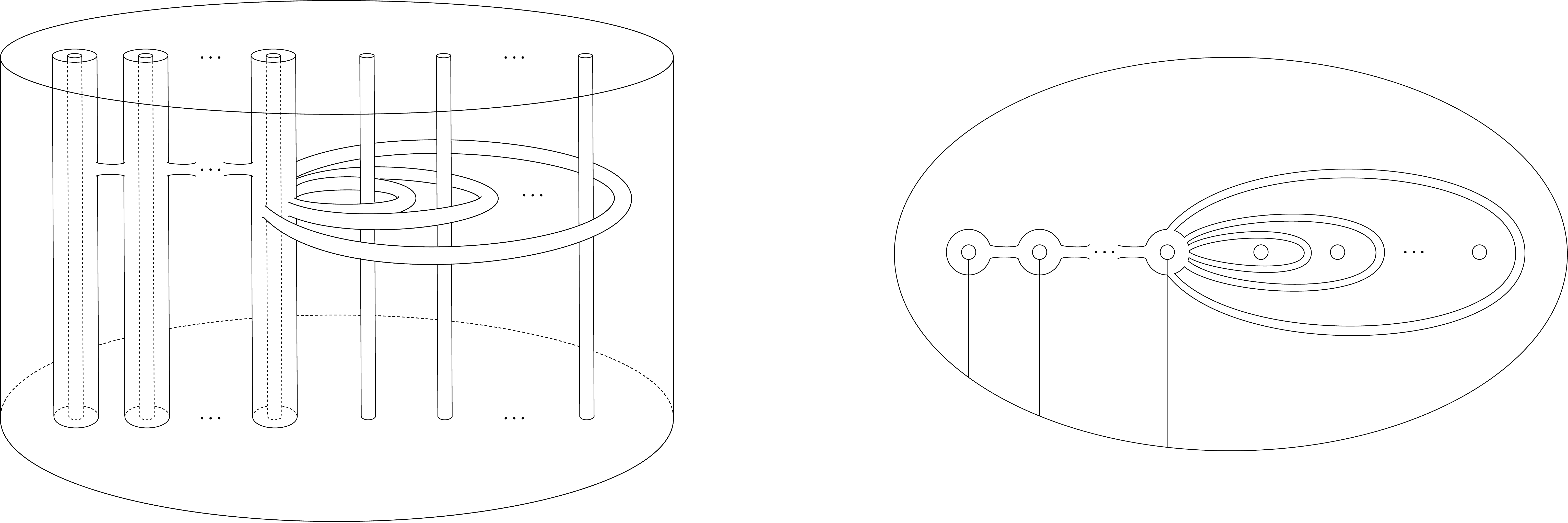}
\caption{Heegaard surface in $D(n)$}
\label{fig:HSofDn}
\end{figure}
\end{rmkk}

In Propositions~3.5 and~3.6 of~\cite{kobayashi-rieck-m-small} 
we studied the behavior of Haken annuli under amalgamation.  
We generalise these propositions as 
Proposition~\ref{pro:haken annulus after amalgamation, general setting} below.
We first explain the construction that is used in Proposition~\ref{pro:haken annulus after amalgamation, general setting}. 
Let $C_{1} \cup_{\s} C_{2}$ be a Heegaard splitting for a manifold $M$ 
that weakly reduces to a (possibly disconnected)
essential surface $F$.   Suppose that $M$ cut open along $F$ consists of two components,
say $M^{(i)}$ $(i=1,2)$.  We denote 
the image of $F$ in $M^{(i)}$ by $F^{(i)}$ and the Heegaard splitting induced on $M^{(i)}$
by $C_{1}^{(i)} \cup_{\s^{(i)}} C_{2}^{(i)}$.  
Suppose that there are Haken annuli for $C_1^{(i)} \cup_{\s^{(i)}} C_2^{(i)}$, say $\mathcal{A}^{(i)}$,
satisfying the following two conditions:  
\begin{itemize}
\item  
there exists a unique component of $\mathcal{A}^{(1)}$, 
say $A^{(1)}$, which intersects 
$F^{(1)}$ in a single simple closed curve, other components are 
disjoint from $F^{(1)}$, and
\item each component of  $\mathcal{A}^{(2)}$ intersects $F^{(2)}$ in a single 
simple closed curve 
isotopic in $F$ to  $A^{(1)} \cap F^{(1)}$.
\end{itemize}
Then let $\widetilde{\mathcal{A}}^{(1)}$ be a collection of mutually disjoint
annuli obtained from $\mathcal{A}^{(1)}$ by substituting $A^{(1)}$ with 
$\vert  \mathcal{A}^{(2)} \vert$ parallel copies 
of $A^{(1)}$ whose boundaries are identified with 
$\mathcal{A}^{(2)} \cap F^{(2)}$.  Finally let  $\widetilde{\mathcal{A}} =
\widetilde{\mathcal{A}}^{(1)} \cup \mathcal{A}^{(2)}$.  Note that
$\widetilde{\mathcal{A}}$ is a system of mutually disjoint annuli  properly
embedded in $M$.  It is easy to adopt the proofs of
Propositions~3.5 and~3.6 of~\cite{kobayashi-rieck-m-small} 
and obtain:

\begin{pro}
\label{pro:haken annulus after amalgamation, general setting}
Let  $M$, $C_1 \cup_\s C_2$, and $\widetilde{\mathcal{A}}$ be as above.  Then
the components of $\widetilde{\mathcal{A}}$ form Haken annuli for $C_1\cup_\s C_2$.
\end{pro}

\section{Various decompositions of knot exteriors}
\label{sec:tfae}

In this section we compare two structures: Hopf-Haken
annuli and $(h, b)$  decompositions.  After defining the two we prove
(Theorem~\ref{thm:tfae}) that they are equivalent.

Let $K$ be a knot in a 3-manifold $M$
and $h \ge 0$, $b \ge 1$ integers. 
We say that $K$ admits a {\it $(h, b)$ decomposition}
(some authors use the term genus $h$, $b$ bridge position)
if there exists a genus $h$ Heegaard 
splitting $C_1 \cup_\s C_2$ of $M$ such that $K \cap C_i$ is a 
collection of $b$ simultaneously boundary parallel arcs 
($i=1,2$; note that  in this paper we do not consider $(h,0)$ decomposition).

Let $K$ be a knot in a compact manifold $M$.  Recall that $E(K)^{(c)}$
is obtained from $E(K)$ by removing $c$ curves that are simultaneously 
isotopic to meridians of $K$.  The trace of the isotopy forms $c$ annuli which
motivates the definition below (Definitions~\ref{dfn:hopf-annuli} and~\ref{dfn:hopf-haken-annuli}
generalize Definition~6.1 of~\cite{kobayashi-rieck-m-small}):

\begin{figure}
  \label{fig:bridge-position}
\end{figure}

\begin{dfn}[a complete system of Hopf annuli]
\label{dfn:hopf-annuli}
Let $K \subset M$ be a knot in a compact manifold and $c > 0$ an integer.
Let $A_{1},\dots,A_{c}$ be annuli disjointly embedded in $E(K)^{(c)}$ so that 
for each $i$, one component of $\del A_{i}$ is a meridian of $\del N(K)$ and 
the other is a longitude of $T_{i}$ (recall $T_{1},\dots,T_{c}$ 
denote the components of $\del E(K)^{(c)} \setminus \del E(K)$).  
Then $\{ A_{1},\dots,A_{c} \}$ is called a {\it complete system of Hopf annuli}.
We emphasize that the complete system of Haken annuli for $E(K)^{(c)}$ 
 is \em not \em unique up-to isotopy.
\end{dfn}

\begin{dfn}[a complete system of Hopf-Haken annuli]
\label{dfn:hopf-haken-annuli}
Let $K \subset M$ be a knot in a compact manifold, $c > 0$ an integer,
$\s$ a Heegaard surface for $E(K)^{(c)}$, and  
$\{ A_{1},\dots,A_{c} \}$ a complete system of Hopf annuli.   
$\{ A_{1},\dots,A_{c} \}$ is called a
{\it complete system of Hopf-Haken annuli for $\s$} if for each $i$, $\s \cap A_i$ is a
single simple closed curve that is essential in $A_i$.
\end{dfn}

\begin{dfn}[Tubing bridge decomposition]
\label{dfn:tubing}
Let $K \subset M$ be a knot in a compact manifold, $\s$ a Heegaard
surface for $E(K)$, and $c > 0$ an integer.  
Suppose that there exists a genus $h-c$ Heegaard surface for $M$ (say $S$)
so that $K$ is $c$ bridge with respect to $S$, and
the surface obtained by tubing $S$ along
$c$ arcs of $K$ cut along $S$ on one side of $S$ is isotopic to $\s$.
Then we say that $\s$ is obtained by {\it tubing $S$ to one side (along $K$)}.
See Figure~\ref{fig:tubing}.
\begin{figure}[h!]
\includegraphics[height=3.5in]{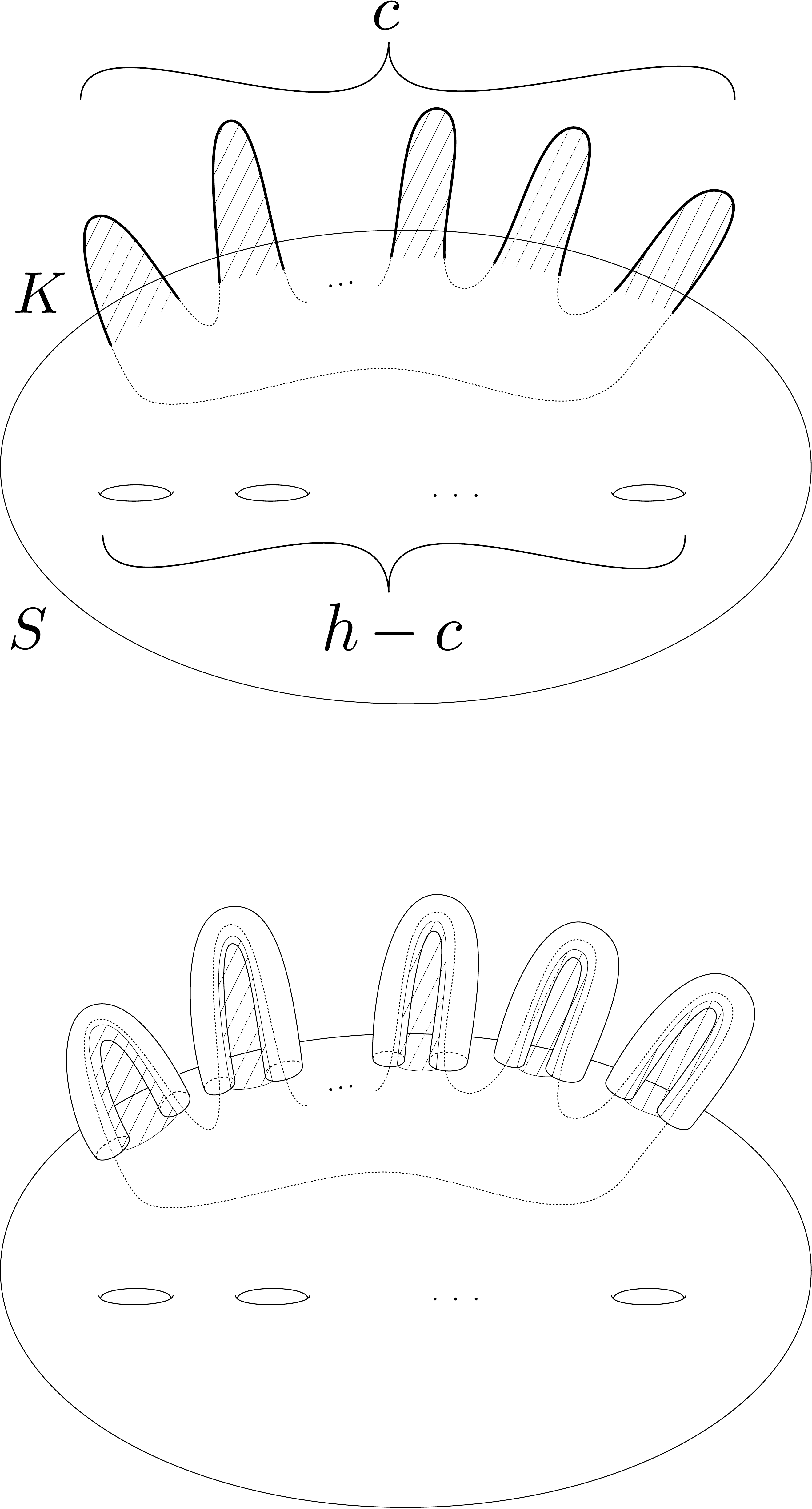}
\caption{Tubing a $(h-c, c)$-decomposition}
\label{fig:tubing}
\end{figure}
\end{dfn}

\begin{thm}
\label{thm:tfae}
Let $M$ be a compact manifold and $K \subset M$ a knot and suppose 
the meridian of $K$ does not bound a disk in $E(K)$.  
Let $c, \ h$ be positive integers.
Then the following two conditions are equivalent:

\begin{enumerate}

\item  $K$ admits an $(h-c, c)$ decomposition.

\item $E(K)^{(c)}$ admits a genus $h$ Heegaard splitting that admits a complete system of Hopf-Haken
  annuli. 
  
\end{enumerate}

\end{thm}

\begin{proof}
\noindent {$(1) \implies (2)$:}  
Let $S \subset M$ be a surface defining a
($h-c$, $c$) decomposition.  Then $S$ separates $M$ into two sides, say
``above'' and ``below''; pick one, say above.  Since the arcs of $K$ above $S$
form $c$ boundary parallel arcs (say $\alpha_1,\dots,\alpha_c$), there are $c$
disjointly embedded disks above $K$ (say $D_1,\dots,D_c$) so that $\del D_i$
consists of two arcs, one  $\alpha_i$ and the other along $S$
(for this proof, see Figure~\ref{fig:tubing}).  Tubing $S$ 
$c$ times along $\alpha_{1},\dots,\alpha_{c}$ 
we obtain a Heegaard surface for $E(K)$ (say $\s$).
We may assume that the tubes are small enough, so that they
intersect each $D_{i}$ in a single spanning arc.
Denote the compression bodies obtained by cutting $E(K)$ along $\s$ by $C_1$ and
$C_2$ with $\del N(K) \subset \partial_{-} C_1$.
Then each $D_i \cap C_2$ is a meridional disk.  Let $A_1,\dots,A_c$ be $c$
meridional annuli properly embedded in $C_1$ 
near the maxima of $K$. 
Then $(\cup_i A_i) \cap \del N(K)$ consists of $c$ meridians, say
$\alpha'_1,\dots,\alpha'_c$.  For each $i$, we isotope $\alpha'_i$ along the
annulus $A_i$ to the curve $A_i \cap \s$ and then push it slightly into
$C_2$, obtaining $c$ curves, 
say $\beta_1,\dots,\beta_c$,  parallel to meridians.  
Drilling $\cup_i \beta_i$ out of $E(K)$
gives $E(K)^{(c)}$.  Using the disks $D_i \cap C_2$ it is easy to see that
$\s$ is a Heegaard surface for $E(K)^{(c)}$.  Clearly, the trace of the isotopy 
from $\cup_{i=1}^n \alpha'_i$ to $\cup_{i=1}^n \beta_i$ 
forms a complete system of Hopf annuli, and
by construction every one of these annuli intersects $\s$ in a single curve
that is essential in the annulus.  
This completes the proof of $(1) \implies (2)$.

\medskip

\noindent {$(2) \implies (1)$:}  
Assume that $E(K)^{(c)}$ admits a Heegaard surface of genus $h$, say $\s$, 
with a complete system of Hopf-Haken annuli, say $\{ A_1, \dots , A_c \}$. 
Let $E(K)' = \mbox{cl}(E(K)^{(c)} \setminus \cup_i N(A_i))$.  Note that $E(K)'$ is 
homeomorphic
to $E(K)$.  Let $S'$ be the meridional surface $\s \cap E(K)'$.  
We may consider $M$ as obtained from $E(K)'$ by meridional Dehn filling 
and $K$ as the core of the attached solid torus. 
By capping off $S'$ we obtain a closed
surface $S \subset M$.  The following claim completes the proof of $(2)
\implies (1)$:

\begin{clm}
\label{clm:ObtainingBridgeDecomposition}
$S$ defines a ($h-c$, $c$) decomposition for $K$.
\end{clm}

\begin{proof}[Proof of claim]
Recall that the components of $\del E(K)^{(c)} \setminus \del E(K)$
were denoted by $T_1, \dots , T_c$, as in Definition~\ref{dfn:hopf-haken-annuli},
so that $A_i \cap T_i \ne \emptyset$ and 
$A_i \cap T_j = \emptyset$ (for $i \ne j$). 
Let $C_1$, $C_2$ be the compression bodies obtained from
$E(K)^{(c)}$ by cutting along $\s$, where $\del N(K) \subset \del_- C_1$.  
Since $\s \cap A_i$ is a single simple closed 
curve which is essential in $A_i$ we have $T_i \subset \del _- C_2$ ($i=1,\dots,c$). 
Denote the annulus $A_j \cap C_i$ by $A_{i,j}$ ($i=1,2$, $j=1, \dots , c$). 

Let $C_i' = C_i \cap E(K)'$ $(i=1,2)$. 
It is clear that $S'$ cuts $E(K)'$ into $C_1'$ and $C_2'$. 
Since $A_i \cap \del N(K)$ is a meridian of $K$, and by assumption
the meridian of $K$ does not bound a disk in $E(K)$, we have that
$A_{i,j}$ is incompressible in $C_i$.
Hence a standard innermost disk, outermost arc argument  
shows that there is a system of meridian disks $\mathcal{D}_i$ of 
$C_i$ which cuts $C_i$ into 
$\del_-C_i \times [0,1]$ such that 
$\mathcal{D}_i \cap (\cup A_{i,j})=\emptyset$. 

Now we consider $C_2$ cut along $\cup A_{2,j}$.
Since $\mathcal{D}_2 \cap (\cup A_{2,j})=\emptyset$, there are components 
$T_1 \times [0,1], \dots , T_c \times [0,1]$ of $C_2$ cut along $\mathcal{D}_2$, 
where $A_{2,j} \subset T_j \times [0,1]$ $(j=1, \dots , c)$. 
Here we note that $T_j \times [0,1]$ cut along $A_{2,j}$ is a solid torus 
in which the image of $T_j \times \{ 0 \}$ is a longitudinal 
annulus (note that the image of $T_j \times \{ 0 \}$ is exactly  $T_j \cap C_2'$).  
This shows that $\{ T_1 \cap C_2', \dots , T_c \cap C_2' \}$ is a primitive system of annuli in $C_2'$, 
that is, there is a system of meridian disks $D_{2,1}, \dots , D_{2,c}$ in
$C_2'$ such that $D_{2,j}\cap (T_j \cap C_2')$ consists of a spanning arc of $T_j \cap C_2'$, and 
$D_{2,j}\cap (T_k \cap C_2') = \emptyset$ $(j \ne k)$. 
Let $C_2''$ be the manifold obtained from $C_2'$ 
by adding $c$ 2-handles along $T_1 \cap C_2', \dots , T_c \cap C_2'$. 
Since $\{ T_1 \cap C_2', \dots , T_c \cap C_2' \}$ is primitive, $C_2''$ is a genus $(h-c)$ compression body, and the union 
of the co-cores of the attached 2-handles, which can be regarded as $K \cap C_2''$, 
are simultaneously isotopic (through the disks $\cup D_{2,j})$ into $\del_+ C_2''$. 

Analogously since
$\mathcal{D}_1 \cap (\cup A_{1,j})=\emptyset$, 
there are $c$ components of $C_1$ cut by 
$\mathcal{D}_1 \cup  (\cup A_{1,j})$  which are 
solid tori such that $\del N(K)$ intersects each solid torus in 
a longitudinal annulus. 
Then the arguments in the last paragraph show that $K \cap C_1''$ 
consists of $c$ arcs which are simultaneously parallel to $S$. 

These show that $S$ gives a $(h-c, c)$ decomposition for $K$, completing the proof of the claim. 
\end{proof}

This completes the proof of Theorem~\ref{thm:tfae}.
\end{proof}

\begin{cor}
\label{cor:upper-bound-for-X(b)}
Let $K$ be a knot in a compact manifold $M$, and suppose that for some 
positive integers $h$ and $c$, $K$ admits a
($h-c$, $c$) decomposition.  Then

$$g(E(K)^{(c)}) \leq h$$
\end{cor}

\begin{proof}
This follows immediately from (1) $\implies$ (2) of Theorem~\ref{thm:tfae}.
\end{proof}

\section{Existence of swallow follow tori and bounding $g(E(K_1 \# \cdots \# K_n)^{(c)})$ above}
\label{sec:sft}

\begin{dfn}[Swallow Follow Torus] 
\label{dfn:sft}
Let $K \subset M$ be a knot and $c\ge 0$ an integer.  An essential separating torus $T \subset E(K)^{(c)}$ is called a
{\it swallow follow torus} if there exists an embedded annulus $A \subset E(K)^{(c)}$ with one
component of $\del A$ a meridian of $E(K)^{(c)}$ and the other an essential curve of $T$,
so that $\mbox{int}(A) \cap T = \emptyset$.
\end{dfn}

In this definition (and throughout this paper) we allow $K$
to be the unknot in $S^3$, in which case $E(K)^{(c)}$ is homeomorphic to a disc with
$c$ holes cross $S^1$, and it admits swallow follow tori whenever $c \geq 3$.

Given a swallow follow torus $T$ and an annulus $A$ as above, we can surger
$T$ along $A$ to obtain a separating meridional annulus.  It is easy to see that since
$T$ is an essential torus, the annulus obtained is essential as well.
Conversely, given an essential separating meridional annulus we can tube the annulus to
itself along the boundary obtaining a swallow follow torus
 (this can be done in two distinct ways).

How does a swallow follow torus decompose a knot exterior?  We first
consider the case $c=0$.  Let $K = K_1 \#
K_2$ be a composite knot (here we are not assuming that $K_1$ or $K_2$
is prime).   Let $\mathcal{A}$ be a decomposing annulus corresponding to the
decomposition of $K$ as $K_1 \# K_2$.  
Thus  $E(K)=E(K_1) \cup_{\mathcal{A} } E(K_2)$. Tubing $\mathcal{A}$ along the boundary (say into
$E(K_2)$) we obtain a swallow follow torus, say $T$.  Clearly, one component
of $E(K)$ cut open along $T$ is homeomorphic to $E(K_2)$.
The other component is homeomorphic to $E(K_1)$ with two meridional
annuli identified, and hence homeomorphic to $E(K_1)^{(1)}$.  Thus we see that a
swallow follow torus $T \subset E(K)$ decomposes $E(K)$ as $E(K_1)^{(1)} \cup_T E(K_2)$.
More generally, given $K,\ K_1,\ K_2$ as above and integers $c,\ c_1,\ c_2
\geq 0$ with $c_1 + c_2 = c$, let $\mathcal{A}$ be a decomposing annulus for
$E(K)^{(c)}$, so that
$E(K)^{(c)} = E(K_1)^{(c_1)} \cup_{\mathcal{A} } E(K_2)^{(c_2)}$.  The swallow
follow torus obtained by tubing $\mathcal{A}$ into $E(K_2)^{(c_2)}$ decomposes
$E(K)^{(c)}$ as $E(K_1)^{(c_1 + 1)} \cup_T E(K_2)^{(c_2)}$.
Since the components of $E(K)^{(c)}$ cut open along a swallow follow torus 
are themselves of the form $E(K_1)^{(c_1 + 1)}$ and $E(K_2)^{(c_2)}$,
we may now extend Definition~\ref{dfn:sft} inductively:
\begin{dfn}[Swallow Follow Tori] 
\label{dfn:sfts}
Let $K$ and $c$ be as in the previous paragraph.
Let $T_1,\dots,T_r$  (for some $r$) be disjointly embedded tori in $E(K)^{(c)}$.
Then $T_1,\dots,T_r$ are called \em swallow follow
tori \em if the following two conditions hold, perhaps after reordering the indices:
\begin{enumerate}
\item $T_1$ is a \sft\ for $E(K)^{(c)}$.
\item For each $i \ge 2$,
$T_i$ is a swallow follow torus for some component of $E(K)^{(c)}$ 
cut open along $\cup_{j=1}^{i-1} T_{j}$.
\end{enumerate}

\end{dfn}

We are now ready to state and prove:

\begin{pro}[Existence of Swallow Follow Tori]
\label{pro:sft}
For $i=1,\dots,n$, let $K_{i}$ be a (not necessarily prime) 
knot in a compact manifold and let $c \geq 0$ be an integer. 
Suppose that  $E(K_{i}) \not\cong T^{2} \times [0,1]$ and
$\partial N(K_{i})$ is incompressible in $E(K_{i})$.

Then given any integers $c_{1},\dots,c_{n} \geq 0$ whose sum is $c+n-1$,
there exist $n-1$ swallow follow tori, denoted $\mathcal{T}$, 
that decompose $E(\#_{i=1}^{n}K_{i})^{(c)}$ as:

$$E(\#_{i=1}^{n}K_{i})^{(c)} = \cup_{\mathcal{T}} E(K_i)^{(c_i)}$$
\end{pro}

\begin{proof}
We use the notation as in the statement of the proposition and induct on $n$.
If $n=1$ there is nothing to prove.  We assume as we may that $n > 1$. 

We first claim that for some $i$ we have that $c_{i} \leq c$.  Assume, for a
contradiction, that $c_{i} > c$ for every $1 \leq i \leq n$.  Since $c_{i}$ and $c$
are integers, $c_{i} \geq c+1$.  Then we have:
$$c+n-1 = \sum_{i=1}^{n} c_{i} \geq n(c+1) = nc + n$$
Moving all term to the right we get that
$$0  \geq  (n-1)c + 1$$
which is absurd, since $n \geq 1$ and $c \geq 0$.
By reordering the indices if necessary we may assume that $c_{n} \leq c$.

Let $A$ be an annulus in $E(\#_{i=1}^{n} K_{i})$ so that the components of 
$E(\#_{i=1}^{n} K_{i})$ cut open along $A$
are identified with $E(K_1 \# \cdots \# K_{n-1})$ and $E(K_{n})$.
Since the tori $\partial N(K_{i})$ are incompressible, $A$ is essential in  $E(\#_{i=1}^{n} K_{i})$.
Recall that  $E(\#_{i=1}^{n} K_{i})^{(c)}$ is obtained from  $E(\#_{i=1}^{n} K_{i})$ by drilling $c$ curves that
are parallel to the meridian; since $c_{n} \leq c$ we may choose the curves
so that exactly $c_{n}$ components are contained in $E(K_{n})$.  
After drilling, the components of $E(\#_{i=1}^{n}K)^{(c)}$ cut open along  $A$ 
are identified with $E(K_1 \# \cdots \# K_{n-1})^{(c-c_{n})}$ and $E(K_{n})^{(c_{n})}$.
Let $T$ be the torus obtained by tubing $A$ into $E(K)^{(c_{n})}$; clearly
the components of  $E(\#_{i=1}^{n}K)^{(c)}$ cut open along  $T$ 
are identified with $E(K_1 \# \cdots \# K_{n-1})^{(c-c_{n}+1)}$ and $E(K_{n})^{(c_{n})}$.
Since $A$ is essential and $E(K_{i}) \not\cong T^{2} \times [0,1]$,
we have the $T$ is essential in $E(\#_{i=1}^{n}K_i)^{(c)}$.  By construction, 
there is an essential curve on $T$ that
cobounds an annulus with a meridian of  $E(\#_{i=1}^{n}K_i)^{(c)}$
and we conclude that $T$ is a swallow
follow torus.

We induct on $K_1,\dots,K_n$.  Let $c' = c-c_n+1$.  Then we have
$$\sum_{i=1}^{n-1} c_{i} = \sum_{i=1}^{n} c_{i} - c_{n} = c + n - 1 - c_{n} = (c - c_{n}+1) + n-2 = c' + (n-1) -1$$
By induction, $E(K_{1}\# \cdots \# K_{n-1})^{(c')}$
admits $n-2$ swallow follow tori, which we will denote by $\mathcal{T}'$,
so that $\mathcal{T}'$ decomposes 
$E(K_{1}\# \cdots \# K_{n-1})^{(c')} = E(K_{1}\# \cdots \# K_{n-1})^{(c - c_{n}+1)}$ as 
$$\cup_{\mathcal{T}'} E(K_{i})^{(c_{i})}$$
It follows that $\mathcal{T} = T \cup \mathcal{T}'$ are swallow follow tori for
$E(K)^{(c)}$, and the components of $E(K)^{(c)}$ cut open along
$\mathcal{T}$ are homeomorphic to 
$E(K_{1})^{(c_{1})},\dots,E(K_{n})^{(c_{n})}$.

\end{proof}

By Proposition~\ref{pro:sft} and repeated application of Lemma~\ref{lem:genus after amalgamation}
we obtain the following.

\begin{cor}
\label{cor:sft}
With notation as in the statement of Proposition~\ref{pro:sft} 
(and in  particular for any integer $c \geq 0$, any integers $c_1,\dots,c_n$
whose sum is $c+n-1$), we get:

$$g(E(K)^{(c)}) \leq \Sigma_{i=1}^n g(E(K_i)^{(c_i)}) - (n-1)$$
\end{cor}

\section{An upper bound for the growth rate}
\label{sec:upper-bound}

Using the results in the previous sections 
we can easily bound the growth rate:

\begin{pro}
\label{pro:upper-bound}
Let $K$ be an admissible knot in a closed manifold $M$.
Let $g = g(E(K)) - g(M)$ and the bridge indices $\{b_1^*,\dots,b_g^* \}$ 
be as in Notation~\ref{notation:bridge indices} in the introduction.

Then 

$$\gr_t(K) \leq \min_{i=1, \dots, g} \bigg\{ 1 - \frac{i}{b_i^*} \bigg\} $$
\end{pro}

\begin{proof}
Fix $1 \leq i \leq g$ and a positive integer $n$.  
Let $k_i>0$ and $0 \leq r < b_{i}^{*}$ be the quotient and remainder
when deviding $(n-1)$ by $b_{i}^{*}$; that is:
$$k_i b^{*}_i + r = n-1$$  
Consider the non-negative integers $b_i^*,\dots,b^*_i,r,0,\dots,0$  
(where $b_i^*$ appears $k_i$ times and the symbol $0$ appears $n-(k_i+1)$ times). 
Applying Corollary~\ref{cor:sft} to $E(nK)^{(0)}$  
we get (recall that $E(nK)^{(0)} = E(nK)$):
$$g(E(nK)) \leq k_i g(E(K)^{(b_i^*)}) + g(E(K)^{(r)}) + (n - (k_i + 1))g(E(K)) - (n-1)$$

By definition of $b_{i}^{*}$, $K$ admits a ($g(E(K))-i, b_i^*$) decomposition. Applying
Corollary~\ref{cor:upper-bound-for-X(b)} with $h - c = g(E(K)) - i$ and $c = b_i^*$ gives
$$g(E(K)^{(b_i^*)}) \leq g(E(K)) - i + b_i^*$$
Thus we get:

\begin{eqnarray*}
g(E(nK)) &\leq& k_i (g(E(K)) - i + b_i^*) + g(E(K)^{(r)}) + (n - (k_i + 1))g(E(K))  -
        (n-1) \\   &=& (n-1)g(E(K)) + g(E(K)^{(r)}) - k_i i + (k_i b_i^* - (n-1)) \\
        &=& (n-1)g(E(K)) + g(E(K)^{(r)}) - k_i i -r
\end{eqnarray*}

By denoting the $n$-th element of the sequence in the definition of the growth
rate by $S_n$, we get:

\begin{eqnarray*}
   S_n  &=& \frac{g(E(nK)) - n g(E(K)) + (n-1)}{n-1} \\  &\leq& \frac{1}{n-1}
        [(n-1)g(E(K)) + g(E(K)^{(r)}) - k_i i -r  - n g(E(K)) + (n-1)]\\   &=&
        \frac{1}{n-1} [g(E(K)^{(r)}) - g(E(K)) - r - k_i i + (n-1)] \\
       &=&  \frac{g(E(K)^{(r)}) - g(E(K)) - r}{n-1} + 1 - 
\frac{k_i i}{k_i b_i^* + r}    
\end{eqnarray*}
In the last equality we used $k_i b^{*}_i + r = n-1$.  Recall that 
$E(K)^{(r)}$ is obtained by drilling $r$ curve parallel to $\del E(K)$ out of
$E(K)$.  Therefore by \cite{rieck}, $g(E(K)^{(r)}) \leq g(E(K)) + r$.  Hence the first
summand above is non-positive, and we may remove that term. 
Furthermore since $r < b_i^*$,  $k_i b_i^* + r <  (k_i+1) b_i^* $, 
which implies 
\begin{equation}
\label{equ:UpperBound}
S_n < 1 - \frac{i}{b_i^*} \frac{k_i}{k_i+1}
\end{equation}
Since $\lim_{n \to \infty} k_i = \infty$ we have:
$$\gr_t(K) = \limsup_{n \to \infty} S_n \leq  
\lim_{k_i \to \infty} \biggl( 1 - \frac{i}{b_i^*} \frac{k_i}{k_i+1}\biggr) = 
1 - \frac{i}{b_i^*}$$ 
As $i$ was arbitrary, we get that
$$\gr_t(K) \leq \min_{i=1, \dots , g} \bigg\{1 - \frac{i}{b_i^*}\bigg\}$$ 
This completes the proof of Proposition~\ref{pro:upper-bound}.
\end{proof}

\part{The growth rate of m-small knots}
\label{part:growth-rate}

This part is devoted to calculating the growth rate of m-small knots,
completing the proof of
Theorem~\ref{thm:main}.  Section~\ref{sec:strongHH} contains the main technical
result of this paper, the Strong Hopf Haken Annulus Theorem
(Theorem~\ref{thm:strongHH}).  This result guarantees the existence of Hopf--Haken
annuli, and complements   Sections~\ref{sec:haken annuli-definitions} and ~\ref{sec:tfae}.  In
Section~\ref{sec:minimal-genus-sft} we prove  existence of  ``special''
swallow follow tori; this section complements Section~\ref{sec:sft}.  Finally,
in Section~\ref{sec:growth-rate} we calculate the growth rate of m-small knots
by  finding  a lower bound that equals exactly the upper bound found in
Section~\ref{sec:upper-bound}.

\section{The Strong Hopf-Haken Annulus Theorem}
\label{sec:strongHH}

Given a knot $K$ in a compact manifold $M$ and an integer $c > 0$; recall that
the exterior of $K$ is denoted by $E(K)$, the manifold obtained by drilling out $c$ curves simultaneously parallel
to the meridian of $E(K)$ is denoted by $E(K)^{(c)}$, 
and the components of $\del E(K)^{(c)} \setminus \del E(K)$
are denoted by $T_{1},\dots,T_{c}$.
Recall also the definitions of  Haken annuli for a given Heegaard splitting
(\ref{dfn:haken annulus}),
a complete system of Hopf annuli (\ref{dfn:hopf-annuli}), 
and a complete system of Hopf-Haken annuli for a given Heegaard splitting (\ref{dfn:hopf-haken-annuli}).

In this section we prove the Strong Hopf Haken Annulus Theorem 
(Theorem~\ref{thm:strongHH}), stated in the introduction. 
Before proving Theorem~\ref{thm:strongHH} we prove three of its main
corollaries:

\begin{cor}
\label{cor:strongHH}
Suppose that the assumptions of Theorem~\ref{thm:strongHH}
are satisfied with $F_1= F_2 = \emptyset$ and in addition, for each $i$, $E(K_{i})$ 
does not admit an essential meridional surface $S$ with 
$\chi(S) \geq \genbd[g(E(K)^{(c)})]$.  Let $h \geq 0$ be an integer.  Then $K$
admits an $(h-c, c)$ decomposition if and only if $g(E(K)^{(c)}) \leq h$.
\end{cor}

\begin{proof}[Proof of Corollary~\ref{cor:strongHH}]
Assume first that $K$ admits an $(h-c, c)$ decomposition.  Then by 
Corollary~\ref{cor:upper-bound-for-X(b)}, we have $g(E(K)^{(c)}) \leq h$.
Note that this direction holds in general
and does not require the assumption about meridional surfaces. 

Next assume that $g(E(K)^{(c)}) \leq h$ and let $\s \subset E(K)^{(c)}$ be 
a genus $h$ Heegaard surface.  
By the assumptions of the corollary, Conclusion~(2) of 
Theorem~\ref{thm:strongHH} does not hold.  Hence by that theorem
$E(K)^{(c)}$ admits a genus $h$ Heegaard surface 
that admits a complete system of Hopf-Haken annuli.  
By $(2) \Rightarrow (1)$ of Theorem~\ref{thm:tfae}, $K$ admits an $(h-c, c)$
decomposition.
\end{proof}

\begin{cor}
\label{cor:strongHH2}
We use the notation of Theorem~\ref{thm:strongHH}.
Suppose that the assumptions of Theorem~\ref{thm:strongHH} hold
and in addition, that each $K_i$ is m-small.
Then for any $c$ and any choice of $F_1$ and $F_2$, 
there is a minimal genus Heegaard splitting of $(E(\#_{i=1}^n K_i)^{(c)};F_1,F_2)$ 
that admits a complete system of Hopf-Haken annuli. 
\end{cor}

\begin{proof}[Proof of Corollary~\ref{cor:strongHH2}]
This follows immediately from Theorem~\ref{thm:strongHH}.
\end{proof}

Next we prove Corollary~\ref{cor:GenusOfXc} which was stated in the introduction:

\begin{proof}[Proof of Corollary~\ref{cor:GenusOfXc}]
We fix the notation in the statement of the corollary.
First we show that for any knot $K$ (not necessarily the connected sum of m-small
knots) if $c \geq b_{g}^{*}$, then 
the inequality $g(E(K)^{(c)}) \leq c$ holds:
by definition of $b_{g}^{*}$, $K$ admits a $(0,b_{g}^{*})$ decomposition 
(recall that 
$K \subset S^{3}$ and hence $b_{g}^*$ is the bridge index of $K$ with respect to $S^2$).
Thus for $c \geq b_{g}^{*}$, $K$ admits a $(0,c)$ decomposition.  By viewing this as a $(c-c,c)$
decomposition, Corollary~\ref{cor:upper-bound-for-X(b)} 
implies that $g(E(K)^{(c)}) \leq c$.

Next we note that  the inequality $g(E(K)^{(c)}) \geq c$ holds for
$K$ that is  a connected sum of  m-small knots, 
and any $c \geq 0$: by Corollary~\ref{cor:strongHH2},
$E(K)^{(c)}$ admits a minimal genus Heegaard surface (say 
$\s$) admitting a complete system of Hopf--Haken annuli.  Hence the $c$ tori, 
$T_1, \dots, T_c$,  
are on the same side of $\s$, 
which implies $g(\s) \geq c$;  hence $g(E(K)^{(c)}) = g(\s) \geq c$.

\end{proof}

\begin{proof}[Proof of Theorem~\ref{thm:strongHH}]
We first fix the notation that will be used in the proof (in addition to the notation in the statement of the theorem).
Let $K$ denote $\#_{i=1}^n K_i$.
For $c > 0$, $E(K)^{(c)}$ admits an essential torus $T$ that decomposes
$E(K)^{(c)}$ as:
$$E(K)^{(c)} = X \cup_T Q^{(c)},$$
where $X \cong E(K)$ and $Q^{(c)} \cong \mbox{ (annulus with }c \mbox{ holes)}\times S^1$.
Note that $Q^{(c)}$ fibers over $S^1$ in a unique way, and the fibers in $T$
are meridian curves in $X \cap Q^{(c)}$.   Since $Q^{(c)}$
is Seifert fibered it is contained in a unique component $J$ of the characteristic
submanifold~\cite{jaco}, \cite{jaco-shalen}, \cite{Johannson}. 
Since $\del N(K_i)$ is incompressible in $E(K_i)$, 
using Miyazaki's result~\cite{miyazaki}
it was shown in~\cite[Claim~1]{kobayashi-rieck-m-small}
that $K$  admits a unique prime decomposition.  Therefore the
number of prime factors of $K$ is well-defined.
We suppose as we may that each knot $K_i$ is prime; consequently,
the integer $n$ appearing in the statement of the theorem is the number of prime factors of $K$.

\begin{figure}[htbp]
        
\caption{}
\end{figure}

\medskip\noindent
{\bf The structure of the Proof.}  The proof is an induction on $(n,c)$ ordered
lexicographically.  We begin with two preliminary special cases.   In Case One we consider 
strongly irreducible Heegaard splittings.  In Case Two we consider
weakly reducible Heegaard splittings so that no component of the essential 
surface obtained by untelescoping is contained in $J$.   
In both cases we prove the theorem directly and
without reference to the complexity $(n,c)$.
We then proceed to the inductive step assuming the theorem for 
$(n',c') < (n,c)$
in the lexicographic order.  By Cases One and Two we may assume that 
a minimal genus Heegaard surface for $E(K)^{(c)}$ is weakly reducible and some
component of the essential surface obtained by untelescoping
it is contained in $J$; this component allows us to induct.  

\bigskip
\noindent 
{\bf Case One. $(E(K)^{(c)}; F_1, F_2)$ admits a strongly irreducible 
minimal genus Heegaard splitting.}    
Let $C_1 \cup_\s C_2$  be a minimal genus strongly
irreducible Heegaard splitting of $(E(K)^{(c)}; F_1, F_2)$.
The Swallow Follow Torus Theorem
\cite[Theorem~4.1]{kobayashi-rieck-m-small} implies that f $n>1$,
either $\s$ weakly reduces to a swallow follow torus (which contradicts the
assumption of Case One)  or Conclusion 2 of Theorem~\ref{thm:strongHH}
holds.  We assume as we may that $n = 1$ in the remainder of the proof of Case One.

Recall the notation $E(K)^{(c)} = X \cup_{T} Q^{(c)}$.
Since $T \subset E(K)^{(c)}$
is essential and
$\s \subset E(K)^{(c)}$ is strongly irreducible, we may isotope $\s$ so that $\s \cap T$
is transverse and every curve of $\s \cap T$ is essential in $T$.
Minimize $|\s \cap T|$ subject to this constraint.  If $\s \cap T = \emptyset$
then $T$ is contained in a compression body $C_{1}$ or $C_{2}$, 
and hence $T$ is parallel to a component of $\partial_- C_{1}$ or $\del_{-}C_{2}$.  
But then $T$ is parallel to a component of 
$\del E(K)^{(c)}$, a contradiction. 
Thus $\s \cap T \neq \emptyset$.  

Let $F$ be a component of $\s$ cut open along $T$.  
Minimality of $|\s \cap T|$ implies that $F$ is not boundary parallel.
Then $\partial F \subset T$; since $T$ is a torus, 
boundary compression of $F$ implies compression into the same side; this will be used extensively
below.  A surface in a Seifert fibered manifold is called {\it vertical}
if it is everywhere tangent to the fibers and {\it horizontal} if it is everywhere transverse to the fibers
(see, for example,~\cite{jaco} for a discussion).  We first reduce Theorem~\ref{thm:strongHH}
as follows:

\begin{ass}
One of the following holds:
\begin{enumerate}
\item $\s \cap X$ is connected and compresses into both sides, and 
$\s \cap Q^{(c)}$ is a collection of essential vertical annuli.
\item Theorem~\ref{thm:strongHH} holds.
\end{enumerate}
\end{ass}

\begin{proof}
A standard argument shows that one component of $\s$ cut open along $T$
compresses into both sides (in $X$ or $Q^{(c)}$) and all other components 
are essential (in $X$ or $Q^{(c)}$); for the convenience of
the reader we sketch it here:
let $D_{1}$ be a compressing disk for $C_{1}$.  After minimizing $|D_{1} \cap T|$
either $D_{1} \cap T = \emptyset$ (and hence some component of $\s$ cut open along $T$
compresses into $C_{1}$) or an outermost disk of $D_{1}$ provides a
boundary compression for some component of $\s$ cut open along $T$; since
boundary compression implies compression into the same side, we see that 
in this case too some component of $\s$ cut open along $T$ compresses into
$C_{1}$.  Similarly, some component of $\s$ cut open along $T$
 compresses into $C_{2}$.  Strong irreducibility
of $\s$ implies that the same component compresses into both sides and all other components are
incompressible and boundary incompressible.  
Minimality of $|\s \cap T|$ implies that no
component is boundary parallel, 
and hence the incompressible and boundary incompressible components are essential.

The proof of Assertion~1 breaks up into three subcases:

\bigskip
\noindent{\bf Subcase 1: no component of $\s \cap X$ is essential.} 
Then $\s \cap X$ is connected and compresses into both sides, and therefore $\s \cap Q^{(c)}$
consists of essential surfaces.  Since $Q^{(c)}$ is Seifert fibered,
every component of $\s \cap Q^{(c)}$ is either horizontal of vertical
(see, for example,~\cite[VI.34]{jaco}).
Any horizontal surface in $Q^{(c)}$ must meet every component of $\del Q^{(c)}$; by
construction $\s \cap \del N(K) = \emptyset$; thus every component of
$\s \cap Q^{(c)}$ is vertical (we will use this argument below without reference).
This gives Conclusion~(1) of the assertion.

\bigskip
\noindent{\bf Subcase 2.a: some component of $\s \cap X$ is essential
and some component of $\s \cap Q^{(c)}$ is essential.}
Let $F$ denote an essential component of $\s \cap X$.  
Since $T$ is incompressible and the components of $\s \cap T$ are essential in $T$,
no component of $\s$ cut open along $T$ is a disk; hence $\chi(F) \ge \chi(\s)$.  
Let $S$ denote an essential component of $\s \cap Q^{(c)}$.  
Then $S$ is a vertical annulus.
In particular, $S \cap T$ consists of fibers in the Seifert fiberation of 
$Q^{(c)}$.  By construction, the fibers on $T$ are meridians of $X$.  We see that
$F$ is meridional, giving Conclusion~(2) of Theorem~\ref{thm:strongHH}.

\bigskip
\noindent{\bf Subcase 2.b: some component of $\s \cap X$ is essential
and no component of $\s \cap Q^{(c)}$ is essential.}
As above let $F$ be an essential component of $\s \cap X$.
By assumption,   no component of $\s \cap Q^{(c)}$
is essential.  Hence
$\s \cap Q^{(c)}$ is connected and compresses into both sides.  
Let $\Delta_{1}$ be a maximal collection of compressing disks for $\s \cap Q^{(c)}$ into $Q^{(c)} \cap C_{1}$
and $S_{1}$ the surface obtained by compressing $S$ along $\Delta_{1}$.
Since $\Delta_1 \neq \emptyset$, maximality of $\Delta_{1}$ and the no nesting lemma~\cite{no-nesting} 
imply that $S_{1}$ is incompressible.  
Suppose first that some non-closed component of $S_{1}$, say $S_{1}'$,
is not boundary parallel (this is similar to Subcase~2.a).
Then $S_{1}'$ is an essential and hence vertical annulus
and we see that $F$ is meridional, giving
Conclusion~(2) of Theorem~\ref{thm:strongHH} and the assertion follows. 
We assume from now on that
$S_{1}$ consists of boundary parallel annuli and, perhaps, closed boundary parallel surfaces and 
ball-bounding spheres.  Furthermore, we see that:
\begin{enumerate}
\item  No two closed components of $S_1$
are parallel to the same component of $\del Q^{(c)}$: this follows from connectivity of $\s \cap Q^{(c)}$
and strong irreducibility of $\s$.
\item No two boundary parallel annuli of $S_1$ are nested: otherwise, it follows from connectivity of $\s \cap Q^{(c)}$
and strong irreducibility of $\s$ that $\s$ can be isotoped out of $Q^{(c)}$; for more details 
see~\cite[Page~249]{kobayashi-rieck-local-det}.  
\end{enumerate}
We assume as we may that the analogous conditions hold after compressing 
$\s \cap Q^{(c)}$ into $Q^{(c)} \cap C_{2}$.
Hence $\s \cap Q^{(c)}$ is a Haagaard surface 
for $Q^{(c)}$ relative to the annuli $\{C_{1} \cap T, C_{2} \cap T\}$
(relative Heegaard surfaces were defined in~\ref{def:RalitiveSplitting}).
We may replace $\s \cap Q^{(c)}$ with the minimal
genus relative Heegaard surface for $Q^{(c)}$ relative to  
$\{C_{1} \cap T$, $C_{2} \cap T\}$ given in Corollary~\ref{cor:RelHSforQc}.  
By pasting this surface to $\s \cap X$
we obtain a closed surface, say $\s'$, fulfilling for following conditions:
\begin{enumerate}
\item  $\s'$ is a Heegaard surface for $E(K)^{(c)}$: 
the components of $X$ cut open along $\s \cap X$ are the same as the components of
$C_1$ and $C_2$ cut open along   $\{C_{1} \cap T, C_{2} \cap T\}$.
Since $T$ is essential, the annuli  $C_{i} \cap T$  are incompressible in $C_i$.
It is well known that cutting a compression body along incompressible
surfaces yields compression bodies; we conclude that
the components of $X$ cut open along $\s \cap X$ are compression bodies.
By definition of relative Heegaard surface, the annuli of
 $\{C_{1} \cap T, C_{2} \cap T\}$ are primitive in the compression bodies obtained by cutting $Q^{(c)}$ open 
 along any relative Heegaard surface; it follows that $E(K)^{(c)}$ cut open alone $\s'$ consists of two compression bodies.
\item $\s'$ is a Heegaard surface for  $(E(K)^{(c)}; F_1, F_2)$:  in addition to~(1) above, we need to show that
 $\s'$ respects the same partition of $\partial E(K)^{(c)} \setminus (\partial N(K),T_{1},\dots,T_{c})$ as $\s$.
This follows immediately from the facts that the changes we made are contained in $Q^{(c)}$, 
every component of $F_1$ is contained in $C_1 \cap X$, and 
every component of $F_2$ is contained in $C_2 \cap X$.
Note that~(1) and~(2) hold for any relative Heegaard surface for $Q^{(c)}$ relative to  
$\{C_{1} \cap T$, $C_{2} \cap T\}$.  
\item $g(\s') = g(\s)$:  minimality of the genus of the relative Heegaard splitting used implies that $g(\s') \leq g(\s)$;
as $\s$ was a minimal genus Heegaard surface for  $(E(K)^{(c)}; F_1, F_2)$, 
$g(\s') = g(\s)$.
Note that~(3) hold for any minimal genus relative Heegaard surface  for $Q^{(c)}$ relative to  
$\{C_{1} \cap T$, $C_{2} \cap T\}$.  
\item $\s'$ admits a completes system of Hopf-Haken annuli:
By Figure~\ref{fig:RelHS} we see directly $\s'$ admits a complete system of Hopf Haken annuli.  
\end{enumerate}
\begin{rmkk}
As noted, in the construction above~(3) holds for any minimal genus relative Heegaard surface.
This is quite different in~(4), when considering Hopf-Haken annuli: it is not hard to construct 
relative Heegaard surfaces that result in a minimal genus Heegaard surface for 
  $(E(K)^{(c)}; F_1, F_2)$ so that all the tori $T_1.\dots,T_c$ are in the compression
  body containing $\del N(K)$, and hence cannot admit even one Hopf-Haken annulus.
  This shows that in the course of the proof of Theorem~\ref{thm:strongHH}
   the given Heegaard surface must be replaced.
\end{rmkk}

The Heegaard surface
$\s'$ fulfils the conditions of Conclusion~(1) of Theorem~\ref{thm:strongHH}. 
This completes that proof of Assertion~1.
\end{proof}

\bigskip
\noindent
Before proceeding we fix the following notation and conventions:
denote $\s \cap X$ by $\s_X$.  By Assertion~1 we may assume that 
$\s_{X}$ is connected and compresses into both sides
and every component of $\s \cap Q^{(c)}$ is an essential vertical annulus.  
Note that $X$ cut open along $\s_X$ consists of exactly two components, denoted
by $C_{i,X}$, where $C_{i,X} = C_i \cap X$  ($i=1,2$).   Denote the collection
of annuli $T \cap C_{i,X}$ by  $\mathcal{A}_i$, and the annuli in
$\mathcal{A}_i$ by $A_{i,1},\dots,A_{i,b}$, where $b$ denotes the number of
annuli in $\mathcal{A}_i$.
We assume from now on that Conclusion~(2) of Theorem~\ref{thm:strongHH} does not hold.

\begin{ass}
The number $b$ fulfils $c \leq b \leq g(\s)$.
\end{ass}

\begin{proof}
Assume for a contradiction that $b < c$.  Since $\s \cap Q^{(c)}$ consists of $b$ annuli,
$Q^{(c)}$ cut open along $\s \cap Q^{(c)}$ consists of $b+1 < c+1$
components.  Hence some component of $Q^{(c)}$ cut open along $\s \cap Q^{(c)}$ 
contains two of the components of $\partial Q^{(c)} \setminus T$.
Hence there is a vatical annulus connecting these components which is
disjoint from $\s$.  Since this annulus is disjoint from $\s$ it is contained in
a compression body $C_i$ and connects two components of $\partial_- C_i$,
which is impossible.  

Since $\s_X$ is obtained by removing the $b$ annuli $\s \cap Q^{(c)}$ and is
connected, $b \leq g(\s)$.

This completes the proof of Assertion~2.
\end{proof}

\begin{ass}
The  surface  
$\s_X$ defines a $(g(\s )-b, b)$ decomposition of $K$. 
\end{ass}

\begin{proof}
For $i=1,2$, let $\Delta_{i}$ be a maximal collection of compressing disks
for $\s_{X}$ into $C_{i,X}$; by assumption, $\Delta_{i} \neq \emptyset$.
Let $S_{i}$ be the surface obtained by compressing $\s_{X}$ along $\Delta_{i}$.
By maximality and the no nesting lemma~\cite{no-nesting} $S_{i}$ is incompressible.
Since the components of $\s \cap Q^{(c)}$ are vertical annuli, the boundary
components of $S_{i}$ are meridians.  Hence, if some non-closed component of $S_{i}$
is essential, we obtain Conclusion~(2) of Theorem~\ref{thm:strongHH},
contradicting our assumption.
Thus $S_{i}$ consists of boundary parallel annuli and, perhaps,
closed boundary parallel surfaces and ball-bounding spheres.
As above, strong irreducibility of $\s$ and connectivity of $\s_X$
imply that these annuli are not nested.
We see that $C_{i,X}$ is a compression body
and $T \cap C_{i,X}$ consists of $b$ mutually primitive annuli.
In fact, we see that $\s_{X}$ is a relative Heegaard surface.
By the argument of Claim~\ref{clm:ObtainingBridgeDecomposition},
on Page~\pageref{clm:ObtainingBridgeDecomposition},
$\s_X$ gives a $(g(\s) - b, b)$ decomposition.
\end{proof}

By Assertion~3 and Theorem~\ref{thm:tfae},
$E(K)^{(b)}$ admits a genus $g(\s)$ Heegaard surface admitting a 
complete system of Hopf-Haken annuli, say $\s'$.  
By Assertion~2, $c \leq b$.  Hence $E(K)^{(c)}$ is obtained from $E(K)^{(b)}$ 
by filling the tori $T_{c+1},\dots,T_b$.  Clearly, $\s'$ is a Heegaard surface for
$E(K)^{(c)}$, admitting a complete system of Hopf-Haken annuli.
This completes the proof of Theorem~\ref{thm:strongHH} in Case One.

\bigskip
\noindent
Before proceeding to Case Two we introduce notation that will be used in that 
case.  
Recall that since $Q^{(c)}$ is Seifert fibered, it is contained in a component of the 
characteristic submanifold  of $E(K)^{(c)}$ denoted by $J$.
Since $X \cong E(K)$ and $K = \#_{i=1}^{n} K_{i}$, $X$ admits $n-1$
decomposing annuli which we will denote by
$A_1,\dots,A_{n-1}$ ($A_{1},\dots,A_{n-1}$ are not uniquely defined).
The components of $X$ cut open along $\cup_{i=1}^{n-1} A_{i}$
are homeomorphic to $E(K_{1}),\dots,E(K_{n})$.  Let $V = Q^{(c)} \cup
N(A_1) \cup \cdots \cup N(A_{n-1})$.  Then $V$ is Seifert fibered and contains
$Q^{(c)}$, and hence after isotopy $V \subset J$.  Note that $V \cap \mbox{cl}(E(K)^{(c)} \setminus V)$
consists of $n$ tori, say $T_{1}',\dots,T_{n}'$.
Finally note that $X^{(c)}$ cut open along $\cup_{i=1}^{n} {T_i}'$ consists of
$n+1$ components, one is $V$, and the remaining homeomorphic
to $E(K_{1}),\dots,E(K_{n})$.  We denote the component 
that corresponds to $E(K_i)$ by $X_{i}$.  After renumbering if
necessary we may assume that $T_{i}'$ is a component of $\del X_{i}$.   By construction 
$T_{i}'$ corresponds to $\del N(K_{i})$.

The proof of the next assertion is a simple argument  using essential arcs in base orbifolds, 
and we leave it to the reader.

\begin{ass}
If $V$ is not isotopic to $J$ then some $E(K_i)$ 
contains a meridional essential annulus.
\end{ass}

For future reference we remark:

\begin{rmkk}
\label{rmk:aboutJ}
By Assertion~4, either we have conclusion 2  of Theorem~\ref{thm:strongHH}, 
or $J=V$. 
Hence, in the following,  we may assume that $J=V$; we will use the notation
$J$ from here on.  By construction, $J$ is homeomorphic to 
($(c+n)$-times punctured disk)$\times S^1$ and hence admits no closed non-separating surfaces.
\end{rmkk}

\medskip
\noindent 
{\bf Case Two. $(E(K)^{(c)};  F_1, F_2)$ admits a weakly reducible minimal genus Heegaard
  surface $\s$, and no component of the essential surface obtained 
  by  untelescoping $\s$ is isotopic into $J$.}

\bigskip
Let $F$ be the (not necessarily connected) essential surface obtained by untelescoping $\s$.
The assumptions of Theorem~\ref{thm:strongHH} imply
that $E(K)^{(c)}$ does not admit a nonseparating  sphere; hence
the Euler characteristic of every component of $F$ is bounded below by
$\chi(\s) + 4$.  After an isotopy that minimizes $|F \cap \del J|$, 
every component of $F \cap J$ is essential in $J$ and every component of 
$F \cap \mbox{cl}(E(K)^{(c)}  \setminus J)$ is essential in $\mbox{cl}(E(K)^{(c)}  \setminus J)$.  
By the assumption of Case Two, 
if some component $F'$ of $F$ meets $J$, then $F' \not\subset J$ and hence
each component of $F' \cap J$ is a vertical annulus and each component of $F'
\cap  \mbox{cl}(E(K)^{(c)} \setminus J)$, say $S$, is a meridional essential surface with $\chi(S) \geq \chi(F'
\cap E(K)^{(c)} ) = \chi(F') \geq \chi(F) \geq \genbd[g(\s)]$, giving Conclusion~2
of Theorem~\ref{thm:strongHH}.  Thus we may assume $F \cap J = \emptyset$.  

Let $M_J$ be the component of $E(K)^{(c)}$ cut open along $F$ containing $J$, 
and let $\s_J$ be the strongly irreducible Heegaard surface induced on $M_J$ by untelescoping.  
Then $\s_J$ defines a partition of $\del M_J \setminus (T_{1} \cup \cdots \cup T_{c} \cup \del N(K))$,
say $F_{J,1}$, $F_{J,2}$.
Since $\s$ is minimal genus, $\s_J$ is a minimal genus
splitting of $(M_J;F_{J,1},F_{J,2})$.

For $i=1,\dots,n$, denote $X_{i} \cap M_J$ by $X_{i}'$.  Note that $X_{i}' \cap J = T_{i}'$;
the meridian of $X_{i}$ 
defines a slope of $T_{i}'$, denoted by $\mu_{i}'$.  By filling $X_{i}'$ along $\mu_{i}'$
we obtain a manifold, say $M_{i}'$, and the core of the attached solid torus
is a knot, say $K_{i}' \subset M_{i}'$.  Then $M_{J}$ is naturally
identified with $E(\#_{i=1}^{n}K_{i}')^{(c)}$, and $\s_{J}$
is a strongly irreducible Heegaard surface for  $(E(\#_{i=1}^{n}K_{i}')^{(c)};F_{J,1},F_{J,2})$.
It is easy to see that $K_i'$ fulfill the assumptions of Theore~\ref{thm:strongHH};
in particular, the assumptions of Case Two imply that $E(K_i') \not\cong T^2 \times I$. 
Therefore, by Case One, one of the following holds:
	\begin{enumerate}
	\item  Conclusion~(2) of Theorem~\ref{thm:strongHH}: for some $i$, $X_{i}'$ admits a meridional essential surface $F_{i}'$ with
	$\chi(F_{i}') \geq {\genbd[g(\s_J )]} \geq {\genbd[g(\s )]}$.
	\item Conclusion~(1) of Theorem~\ref{thm:strongHH}: there exists a Heegaard surface $\s_{J}'$ for
	$M_{J}$ so that the following three conditions hold:
		\begin{enumerate}
		\item $g(\s_{J}') = g(\s_{J})$, 
		\item $\s_{J}'$ is a Heegaard splitting for $(E(\#_{i=1}^{n}K_{i}')^{(c)};F_{J,1},F_{J,2})$,
		\item $\s_{J}'$ admits a complete system of Hopf--Haken annuli.
		\end{enumerate}
	\end{enumerate}
Assume first that~(1) holds.  Since $X_{i}'$ is a component of $X_{i}$ cut open along
the (possibly empty) surface $F \cap X_i$, and every component of $F \cap X_i$ is incompressible,
we have that
$F_{i}'$ is essential in $X_{i}$.  By construction,
the meridians of $X_{i}$ and $X_{i}'$ are the same.  Finally, $\chi(F_{i}') \ge \genbd[g(\s )] 
=\genbd[g(E(K)^{(c)}; F_1, F_2)]$.  
This gives Conclusion~2 of Theorem~\ref{thm:strongHH}.

Assume next that~(2) happens.  
By condition~(2)(b), $\s_J'$ induces the same partition on the components of
$\partial M_j \setminus \{T_1,\dots,T_c,\partial N(K)\}$ as $\s_J$.
Thus we may amalgamate the Heegaard surfaces 
induced on the components of $\mbox{cl}(E(K)^{(c)} \setminus M_{J})$ with $\s_J'$, obtaining a 
Heegaard surface for $(E(K)^{(c)};F_1,F_2)$, say $\s''$. 
By Proposition~\ref{pro:haken annulus after amalgamation, general setting}, 
$\s''$ admits a complete system of 
Hopf--Haken annuli.  Since $g(\s_J') = g(\s_{J})$, we have that $g(\s'') = g(\s)$; hence $\s''$ is a minimal 
genus Heegaard surface for $(E(K)^{(c)};F_1,F_2)$.
This gives Conclusion~1 of Theorem~\ref{thm:strongHH}, completing 
the proof of Theorem~\ref{thm:strongHH} in Case Two.

\bigskip\noindent  
With these two preliminary cases in hand we are now ready for the inductive step.
For the remainder of the proof we assume that 
Conclusion~2 of Theorem~\ref{thm:strongHH} does not hold. 
Fix $K_1,\dots,K_n$ and $c \geq 0$ and assume, by induction, that
Theorem~\ref{thm:strongHH} holds for any example with complexity $(n',c') < (n,c)$
ordered lexicographically.
Let $\s$ be a minimal genus Heegaard surface for 
$E(\#_{i=1}^{n} K_{i})^{(c)}$.  By Case One, we may assume that
$\s$ is not strongly irreducible; hence $\s$ admits an untelescoping.  
By Case Two, we may assume that some component $F'$
of the essential surface $F$ obtained by untelescoping $\s$ is isotopic into $J$.
By Remark~\ref{rmk:aboutJ}, $J$ is a Seifert fibered space over a punctured disk 
and the components of $E(\# K_i)^{(c)} \setminus J$ are identified with $E(K_1),\dots,E(K_n)$. 
After isotopy we may assume that $F'$ is horizontal
or vertical (see, for example, \cite[VI.34]{jaco}; recall that a surface in a Seifert fibered space is 
horizontal  if it is everywhere transverse to the fibers and vertical if it is everywhere tangent to the fibers).
However $\del J \neq \emptyset$
and $\del F' = \emptyset$, and therefore $F'$ cannot be horizontal.  We conclude
that $F'$ is a vertical torus that separates $J$ and hence $E(\#_{i=1}^{n} K_{i})^{(c)}$.
Thus $F'$ decomposes $E(\#_{i=1}^n K_i)^{(c)}$ as:
$$E(\#_{i=1}^n K_i)^{(c)} = E(\#_{i \in I} K_i)^{(c_1)} \cup_{F'} E(\#_{i \not\in I} K_i)^{(c_2)},$$
where $c_1 + c_2 = c+1$ and $I \subset \{1,\dots,n\}$.  Since $F'$ is connected and separating, 
by Proposition~\ref{pro:untel to a conn sep surfce implies weak reduction}
$\s$ weakly reduces to $F'$ and the weak reduction induces (not necessarily
strongly irreducible) Heegaard splittings on 
$E(\#_{i \in I} K_i)^{(c_1)}$ and $E(\#_{i \not\in I} K_i)^{(c_2)}$. 
We divide the proof into Cases~1 and~2 below:

\bigskip
\noindent{\bf Case 1: $I = \emptyset$ or $I = \{1,\dots,n\}$.} 
By symmetry we may assume that $I = \{1,\dots,n\}$.  
Then $F'$ decomposes $E(\#_{i=1}^n K_i)^{(c)}$ 
as  $E(\#_{i=1}^n K_i)^{(c_1)} \cup_{F'} D(c_2)$ 
where $D(c_{2})$ is a $c_{2}$ times punctured disk cross $S^1$.
There are two possibilities:  $\del N(K) \subset E(\#_{i =1}^n K_{i})^{(c_{1})}$
(Subcases~1.a) and $\del N(K) \subset D(c_2)$ (Subcases~1.b).

\begin{figure}[h!]
\includegraphics[height=1.4in]{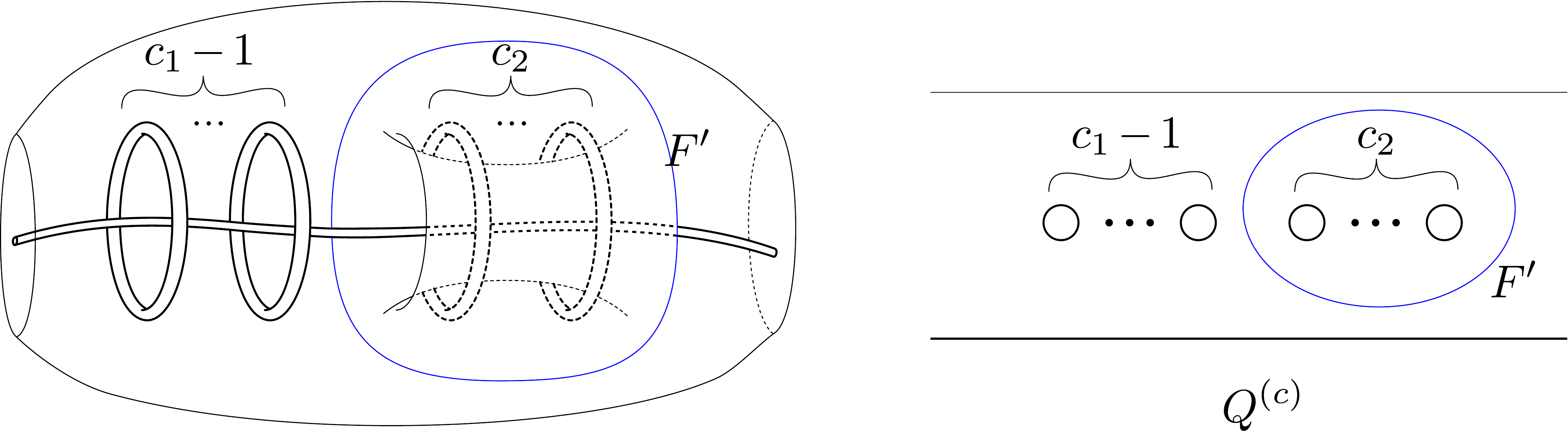}
\caption{Case 1.a}
\label{fig:Case1a}
\end{figure}

\bigskip
\noindent{\bf Subcase 1.a: $I = \{1,\dots,n\}$ and $\del N(K) \subset E(\#_{i = 1}^n K_{i})^{(c_1)}$.}  
For this subcase, see Figure~\ref{fig:Case1a}.  
Recall that $E(\#_{i=1}^n K_i)^{(c)} = E(\#_{i=1}^n K_i)^{(c_1)} \cup_{F'} D(c_2)$ with $c_1+c_2=c+1$;
by reordering $T_1,\dots,T_c$ if necessary we may assume that 
$T_1,\dots,T_{c_1-1} \subset \del E(\#_{i=1}^n K_i)^{(c_1)}$ and $T_{c_1},\dots,T_c \subset \del D(c_2)$.
Since $F'$ is not boundary parallel
$c_2 \geq 2$; thus $c_1 < c$.  
Thus $(n,c_1) < (n,c)$ (in the lexicographic order) and hence we may apply induction
to $E(\#_{i=1}^n K_i)^{(c_1)}$.  Let $\s_1'$ be the Heegaard surface induced on $E(\#_{i=1}^n K_i)^{(c_1)}$ by 
the weak reduction of $\s$.
By assumption Conclusion~2 of  Theorem~\ref{thm:strongHH}
does not hold; it is easy to see that $E(\#_{i=1}^n K_i)^{(c_1)}$ fulfills the assumptions of Theorem~\ref{thm:strongHH},
and since $g(\s_{1}') < g(\s)$, Conclusion~(2) does not hold for $E(\#_{i=1}^n K_i)^{(c_1)}$. 
Therefore the inductive hypothesis shows that $E(\#_{i=1}^n K_i)^{(c_1)}$
admits a Heegaard surface $\s_1$ fulfilling the following three conditions:
\begin{enumerate}
\item $g(\s_{1}) = g(\s_{1}')$;
\item $\s_{1}$ and $\s_{1}'$ induces the same partition of the components of

$\partial E(\#_{i = 1}^n K_{i})^{(c_1)} \setminus \{T_1,\dots,T_{c_1-1},F',\partial N(K)\}$;
\item $\s_{1}$ admits a complete system of Hopf--Haken annuli.
\end{enumerate}
Denote the union of the $c_1-1$ 
Hopf--Haken annuli connecting $\del N(\#_{i=1}^{n} K_{i})$ to $T_1,\dots,T_{c_1-1}$
by $\mathcal{A}_1$ and the Hopf--Haken annulus connecting $\del  N(\#_{i=1}^{n} K_{i})$ to $F'$ by $A$
(note that $c_1-1=0$ is possible; in that case $\mathcal{A}_1 = \emptyset$).  
There exists a minimal genus Heegaard surface $\s_2$ for $D(c_2)$ 
that admits $c_2$    Haken
annuli $A_{c_{1}},\dots,A_{c}$ so that one component of $\del A_i$ 
is a longitude of $T_i$ and the other is on $F'$ and parallel to $A \cap F'$ there
(recall Remark~\ref{rmk:HakenAannuliForDC}).  
We denote $\cup_{i=c_{1}}^{c} A_{i}$ by $\mathcal{A}_{2}$.
As shown in Proposition~\ref{pro:haken annulus after amalgamation, general setting},
the annuli obtained by attaching a parallel copy of $A$ to 
each annulus of $\mathcal{A}_2$ union $\mathcal{A}_1$ 
are Haken annuli
for the Heegaard surface obtained by amalgamating $\s_1$ and $\s_2$;
we will denote this surface by $\hat \s$.  By construction, these annuli form a 
complete system of Hopf-Haken annuli for $\hat \s$. 
Since $g(\hat \s) = g(\s_{1}) + g(\s_{2}) - 1$ and  $g(\s) = g(\s_{1}') + g(\s_{2}) - 1$,
by Condition~(1) above we have $g(\hat{\s})=g(\s)$.  By construction $\s$ and $\s'$ induce the
same partition of the components of $\partial E(K)^{(c)} \setminus \{T_{1},\dots,T_{c},\partial N(K)\}$.
Theorem~\ref{thm:strongHH} follows in Subcase~1.a.

\bigskip
\noindent{\bf Subcase 1.b: $I = \{1,\dots,n\}$ and $\del N(K) \subset D(c_{2})$.}  
\begin{figure}[h!]
\includegraphics[height=1.4in]{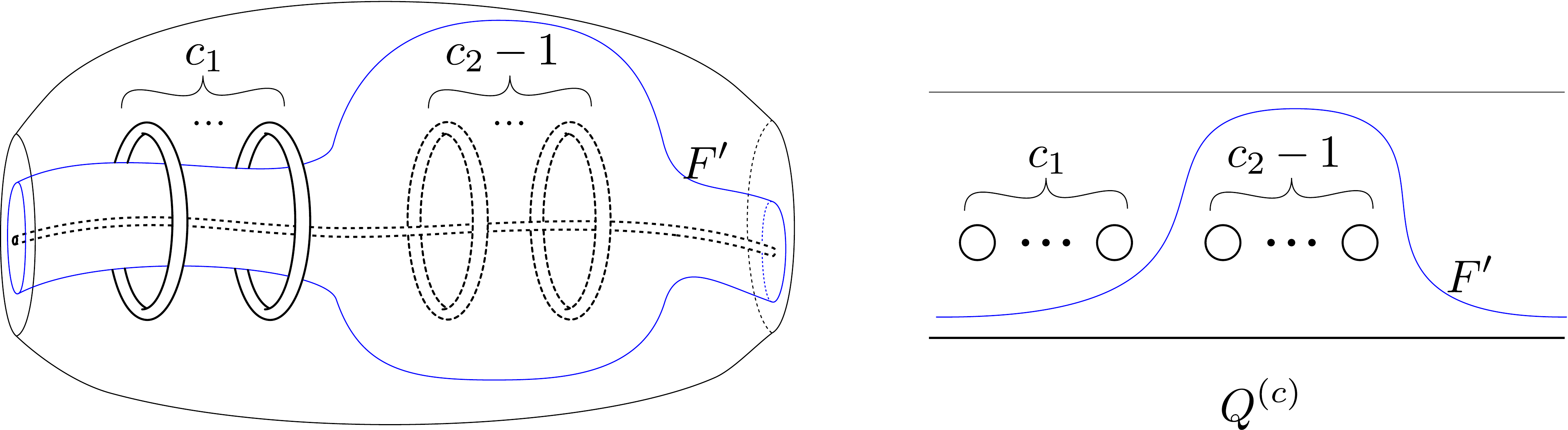}
\caption{Subcase 1.b}
\label{fig:Case1b}
\end{figure}
For this subcase see Figure~\ref{fig:Case1b}.  Since Subcase~1.b is similar to 
Subcase~1.a we omit some of the easier details of the proof.  
As in Subcase~1.a, $F'$ decomposes $E(\#_{i=1}^n K_i)^{(c)}$ 
as  $E(\#_{i=1}^n K_i)^{(c_1)} \cup_{F'} D(c_2)$ with $c_1+c_2=c+1$; 
we reorder $T_1,\dots,T_c$ so that $T_1,\dots,T_{c_1} \subset \del E(\#_{i=1}^n K_i)^{(c_1)}$
and $T_{c_1+1},\dots,T_c \subset \del D(c_2)$.  By induction there exists a minimal genus Heegaard 
surface $\s_{1}$ for $E(\#_{i=1}^n K_i)^{(c_1)}$ fulfilling conditions
analogous to~(1)--(3) listed in Subcase~1.a.
In particular, $\s_{1}$ admits a complete system of $c_1$ Hopf--Hakn annuli,
say $\mathcal{A}_1$, so that one boundary component of each annulus of $\mathcal{A}_{1}$
is a longitude of $T_{i}$ ($i=1,\dots,c_{1}$) and the other is a curve of $F'$.  As in Subcase~1.a, 
there exists a minimal genus Heegaard surface $\s_{2}$ for 
$D(c_2)$ admitting a system of $c_{2}$ Haken
annuli (recall Remark~\ref{rmk:HakenAannuliForDC}), 
denoted by $\mathcal{A}_{2} \cup A$,
so that $\mathcal{A}_{2}$ consists of $c_{2} - 1$ annuli connecting meridians of $\del N(\# K_{i})$ to
the longitudes of $T_{c_{1}+1},\dots,T_{c}$, and $A$ connects a meridian of $\del N(\# K_{i})$ to a
curve of $F'$; by construction, this curve is parallel to the curves of $\mathcal{A}_{1} \cap F'$.
As shown in Proposition~\ref{pro:haken annulus after amalgamation, general setting},
the annuli obtained by attaching a parallel copy of $A$ to 
each annulus of $\mathcal{A}_1$ union $\mathcal{A}_2$ 
are Haken annuli
for the Heegaard surface obtained by amalgamating $\s_1$ and $\s_2$;
we will denote this surface by $\hat \s$.  By construction, these annuli form a 
complete system of Hopf-Haken annuli for $\hat \s$. 
As in Case~1.a, $g(\hat{\s}) = g(\s)$ and $\hat \s$ induces the same partition 
on the components of  $\partial E(K)^{(c)} \setminus \{T_{1},\dots,T_{c},\partial N(K)\}$ as $\s$. 
Theorem~\ref{thm:strongHH} follows in Subcase~1.b.

\bigskip
\noindent{\bf Case 2: $\emptyset \neq I \not= \{1,\dots,n\}$.}  
See Figure~\ref{fig:Case2} for this case.
\begin{figure}[h!]
\includegraphics[height=1.4in]{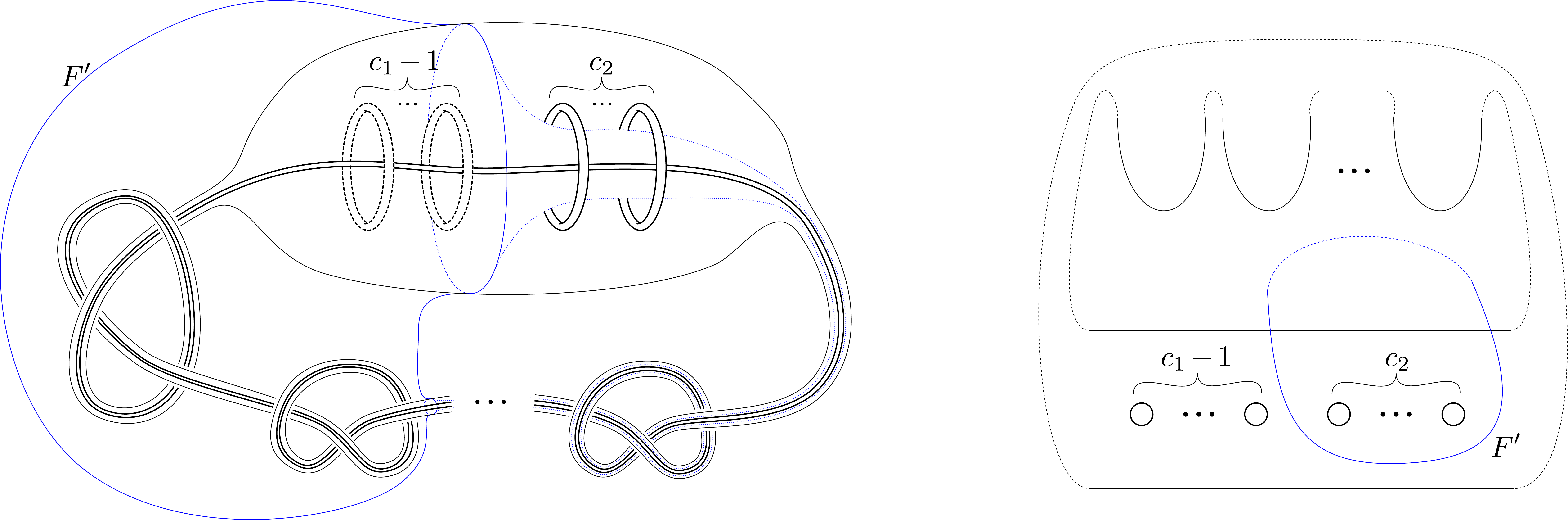}
\caption{Case 2}
\label{fig:Case2}
\end{figure}
Since Case~2 is similar to 
Subcase~1.a we omit some of the easier details of the proof.
By symmetry we may assume that $\del N(K) \subset \del E(\#_{i \in I} K_i)^{(c_{1})}$.   
Let $\s_{1}'$ and $\s_{2}'$ be the Heegaard surfaces induced on 
$E(\#_{i \in I} K_{i})^{(c_{1})}$ and $E(\#_{i \not\in I} K_{i})^{(c_{2})}$
(respectively) by $\s$.
Since both $|I|$ and
$n - |I|$ are strictly less than $n$, we may apply induction to both 
$E(\#_{i \in I} K_{i})^{(c_{1})}$ and $E(\#_{i \not\in I} K_{i})^{(c_{2})}$.
By induction, there exists a minimal genus Heegaard surfaces $\s_{1}$ and $\s_{2}$ for
$E(\#_{i \in I} K_{i})^{(c_{1})}$ and $E(\#_{i \not\in I} K_{i})^{(c_{2})}$
(respectively) fulfilling the following three conditions:
\begin{enumerate}
\item $g(\s_{1}) = g(\s_{1}')$ and $g(\s_{2}) = g(\s_{2}')$;
\item $\s_{1}$ induces the same partition of the components of \\
$\partial E(\#_{i \in I} K_{i})^{(c_{1})} \setminus \{\partial N(K), T_{1},\dots,T_{c_{1}-1}\}$  as $\s_{1}'$;
similarly, $\s_{2}$ induces the same partition of the components of
$\partial E(\#_{i \not\in I} K_{i})^{(c_{2})} \setminus \{T_{c_{1}},\dots,T_{c}\}$ as $\s_{2}'$.
\item $\s_{1}$ admits a complete system of Hopf-Haken annuli, say $A \cup \mathcal{A}_1$,
where $A$ connects $\partial N(K)$ to $F'$ and the components of $\mathcal{A}_1$
connect $\partial N(K)$ to $T_1,\dots,T_{c_1 - 1}$; similary 
$\s_{2}$ admit complete systems of Hopf--Haken annuli $\mathcal{A}_2$ whose
components connect $F'$ to $T_{c_1},\dots,T_{c}$.
\end{enumerate}
As shown in Proposition~\ref{pro:haken annulus after amalgamation, general setting},
the annuli obtained by attaching a parallel copy of $A$ to 
each annulus of $\mathcal{A}_2$ union $\mathcal{A}_1$ 
are Haken annuli
for the Heegaard surface obtained by amalgamating $\s_1$ and $\s_2$;
we will denote this surface by $\hat \s$.  By construction, these annuli form a 
complete system of Hopf-Haken annuli for $\hat \s$. 
As above $g(\hat{\s}) = g(\s)$ and $\hat \s$ induces the same partition of the components
of $\partial E(K)^{(c)} \setminus \{T_{1},\dots,T_{c},\partial N(K)\}$ as $\s$.
Theorem~\ref{thm:strongHH} follows in Case~2.

This completes the proof of Theorem~\ref{thm:strongHH}.
\end{proof}

\section{Weak reduction to swallow follow tori and calculating $g(E(K)^{(c)})$}
\label{sec:minimal-genus-sft}

Let $K_1\subset M_{1},\dots,K_n \subset M_{n}$ be knots in compact manifolds 
and $c>0$ an integer.  When convenient, we will denote $\#_{i=1}^{n} K_{i}$ by $K$.
Let $c_1, \dots, c_n \geq 0$ be integers such that $\sum_{i=1}^{n} c_i = c + n - 1$.   
By Proposition~\ref{pro:sft}  
there exist $n-1$ swallow follow tori $\mathcal{T} \subset E(K)^{(c)}$ that decompose
it as $E(K)^{(c)} = \cup_{\mathcal{T}} E(K_{i})^{(c_i)}$.   By amalgamating minimal
genus Heegaard surfaces for $E(K_{i})^{(c_i)}$ we obtain a Heegaard
surface for $E(K)^{(c)}$; however, it is distinctly possible that the surface
obtained is not of minimal genus.  This motivates the following definition:

\begin{dfn}[natural swallow follow tori]
Let $K_{1}\subset M_{1},\dots,K_{n} \subset M_{n}$
be prime knots in compact manifolds and $c \geq 0$ an integer.
Let $\mathcal{T} \subset E(\#_{i=1}^{n} K_{i})^{(c)}$ be a collection of $n-1$ 
swallow follow tori giving the decomposition 
$E(\#_{i=1}^{n} K_{i})^{(c)} = \cup_\mathcal{T} E(K_{i})^{(c_i)}$, for some integers $c_{i} \geq 0$.
We say that $\mathcal{T}$ is {\it natural} if it is
obtained from a minimal genus Heegaard surface for 
$E(\#_{i=1}^{n} K_{i})^{(c)}$ by iterated weak reduction;
equivalently, $\mathcal{T}$ is called natural if 
$$g(E(\#_{i=1}^{n} K_{i})^{(c)}) = \sum_{i=1}^n g(E(K_{i})^{(c_i)}) - (n-1)$$
\end{dfn}

\begin{rmk}
As explained in Section~\ref{sec:sft}, given {\it any} collection of $n-1$
swallow follow tori $\mathcal{T} \subset E(\#_{i=1}^{n} K_{i})^{(c)}$ that give the decomposition 
$E(\#_{i=1}^{n} K_{i})^{(c)} = \cup_\mathcal{T} E(K_{i})^{(c_i)}$, 
the integers $c_1,\dots,c_n$ satisfy the equation $\sum_{i=1}^{n} c_i = c+n-1$.
We will often use this fact without reference; compare this to Proposition~\ref{pro:sft}
where the converse was established.
\end{rmk}

\begin{example}[Knots with no natural swallow follow tori]
\label{ex:sft1}
In Theorem~\ref{thm:sft} below we prove existence of natural
swallow follow tori under certain assumptions.  
The following example shows that this is not always the case.
We first analyze basic properties of knots that admit natural swallow follow tori:
let $K_1, K_2 \subset S^3$ be prime knots and $T \subset E(K_1 \# K_2)$ a natural swallow
follow torus.  By exchanging the subscripts if necessary we may assume that $T$ decomposes
$E(K_{1} \# K_{2})$ as $E(K_{1})^{(1)} \cup_{T} E(K_{2})$. By definition of naturallity,
$$g(E(K_{1} \# K_{2})) =  g(E(K_{1})^{(1)}) + g(E(K_{2})) - 1$$  
It is easy to see that $g(E(K_{1})^{(1)}) \geq g(E(K_{1}))$.  Combining these,
we see that $g(E(K_{1} \# K_{2})) \geq  g(E(K_{1})) + g(E(K_{2})) - 1$.
Morimoto~\cite{morimoto-proc-ams} constructed examples of prime knots $K_{1}$, $K_{2}$
for which $g(E(K_{1} \# K_{2})) =  g(E(K_{1})) + g(E(K_{2})) - 2$.  We conclude that for these knots,
$E(K_{1} \# K_{2})$ does not admit a natural swallow follow torus.
\end{example}

\begin{example}[knots where only certain swallow follow tori are natural]
\label{ex:sft2}
The following example is of a more subtle phenomenon.  It shows 
that even when $E(K_{1} \# K_{2})$ does admit a natural
swallow follow torus, not every swallow follow torus
is natural.  In this sense, the weak reduction found in
Theorem~\ref{thm:sft} is special as it finds natural swallow follow tori.

Let $K_{MSY} \subset S^{3}$ be the knot constructed by Morimoto
Sakuma and Yokota in \cite{skuma-morimoto-yokata} and recall the notation
$2K_{MSY} = K_{MSY} \# K_{MSY}$.
It was shown in \cite{skuma-morimoto-yokata} that $g(E(K_{MSY})) = 2$ and
$g(E(2K_{MSY}))=4$.  

We claim that $g(E(K_{MSY})^{(1)})=3$.
By \cite{rieck}, either $g(E(K_{MSY})^{(1)})=2$ or
$g(E(K_{MSY})^{(1)})=3$.  Assume for a contradiction that $g(E(K_{MSY})^{(1)})=2$.
By Corollary~\ref{cor:sft} (with $c=0$, $c_1=1$, and $c_2=0$) we have
$$g(E(2K_{MSY})) \leq g(E(K_{MSY})^{(1)}) + g(E(K_{MSY})) - 1 = 2+2-1=3$$ 
a contradiction.  Hence $g(E(K_{MSY})^{(1)}) = 3$.

Let $K$ be any non-trivial 2-bridge knot.  
It is well known that $g(E(K))=2$.
We claim that $g(E(K_{MSY} \# K)) =3$.  
Since tunnel number one knots are prime~\cite{norwood},
$g(E(K_{MSY} \# K)) \geq 3$.  On the other hand, since $K$
admits a $(1, 1)$ decomposition, by Theorem~\ref{thm:tfae} we have that $g(E(K)^{(1)}) = 2$. 
As above, Corollary~\ref{cor:sft} gives
$$g(E(K_{MSY} \# K)) \leq g(E(K_{MSY})) + g(E(K)^{(1)}) - 1 = 2 + 2 - 1 = 3$$   
Hence $g(E(K_{MSY} \# K)) =3$.

$E(K_{MSY} \# K)$ admits two swallow follow tori, say $T_1$ and $T_2$,  
that decompose it as follows:
\begin{enumerate}
\item $g(E(K_{MSY} \# K))  = E(K_{MSY})^{(1)} \cup_{T_1} E(K)$, and
\item $g(E(K_{MSY} \# K)) = E(K_{MSY}) \cup_{T_2} E(K)^{(1)}$.
\end{enumerate}
In each case, amalgamating minimal genus Heegaard surfaces for the manifolds appearing on the
right hand side yields a Heegaard surface for $E(K_{MSY} \# K)$ whose genus fulfills (Lemma~\ref{lem:genus after amalgamation}):
\begin{enumerate}
\item $g(E(K_{MSY})^{(1)}) + g(E(K)) - g(T_1) = 3 + 2 - 1 = 4,$ and
\item $g(E(K_{MSY})) + g(E(K)^{(1)}) - g(T_2) = 2 + 2 - 1 = 3.$
\end{enumerate}
We conclude that $T_{2}$ is a natural swallow follow torus but $T_1$ is not.
\end{example}

\bigskip
\bigskip
\noindent In this section we show that if $K_i$ is
m-small for all $i$, then any minimal genus Heegaard surface for $E(\#_{i=1}^{n} K_{i})^{(c)}$ weakly reduces to a
natural collection of swallow follow tori.  The statement of Theorem~\ref{thm:sft} is more
general and allows for non-minimal genus Heegaard surfaces.

\begin{thm}
\label{thm:sft}
Let $K_{i} \subset M_{i}$ be prime knots in compact manifolds 
so that $E(K_i)$ not homeomorphic to 
$T^2 \times I$, $E(K_i)$ is irreducible, and $\del N(K_i)$ is incompressible in $E(K_i)$. 
Let $\s$ be a (not necessarily minimal genus) Heegaard surface for
$E(\#_{i=1}^{n} K_{i})^{(c)}$. Then one of the following holds:

\begin{enumerate}
\item $\s$ admits iterated weak reductions 
that yield a collection of $n-1$ swallow follow tori, say  $\mathcal{T}$, giving the decomposition 
$$E(\#_{i=1}^{n} K_{i})^{(c)} = \cup_\mathcal{T} E(K_{i})^{(c_i)}$$
where $c_{1},\dots,c_{n}$  are integers such that $\s_{i=1}^{n} c_i = c+n - 1$.  

\item For some $i$, $K_i$ admits an essential meridional surface $S$ with
  $\chi (S) \geq \genbd[g( \s )]$.  
\end{enumerate}
\end{thm}

The main corollary of Theorem~\ref{thm:sft} allows us to calculate $g(E(\#_{i=1}^{n} K_{i})^{(c)})$
in terms of $g(E(K_{i})^{(c_{i})})$.

\begin{cor}
\label{cor:sft-strong}
In addition to the assumptions of Theorem~\ref{thm:sft}, suppose that no $K_i$ admits an
essential meridional surface $S$ with $\chi(S) \geq \genbd[g(E(\#_{i=1}^{n} K_{i})^{(c)})]$.  Then 
$E(\#_{i=1}^{n} K_{i})^{(c)}$ admits a natural collection of $n-1$ swallow follow tori; 
equivalently, there exist integers $c_1, \dots, c_n \geq 0$ so that $\sum_{i=1}^{n} c_i = c + n - 1$ and 
$$g(E(\#_{i=1}^{n} K_{i})^{(c)}) = \s_{i=1}^n g(E(K_{i})^{(c_{i})}) - (n-1)$$
\end{cor}

\begin{proof}
Apply Theorem~\ref{thm:sft} to a minimal genus Heegaard splitting of $E(\#_{i=1}^{n}K_i)^{(c)}$
and apply Lemma~\ref{lem:genus after amalgamation}.
\end{proof}

We get:

\begin{cor}
\label{cor:sft-strong2}
In addition to the assumptions of Theorem~\ref{thm:sft}, suppose that no $K_i$ admits an
essential meridional surface $S$ with $\chi(S) \geq \genbd[g(E(\#_{i=1}^{n} K_{i})^{(c)})]$.  Then 
$$g(E(\#_{i=1}^{n} K_{i})^{(c)}) = \min\Big\{\sum_{i=1}^n g(E(K_{i})^{(c_{i})}) - (n-1)\Big\}$$
where the minimum is taken over all integers $c_1, \dots, c_n \geq 0$ with $\s  c_i = c + n - 1$.
\end{cor}

\begin{proof}
By  Corollary~\ref{cor:sft},
for any collection of integers $c_1,\dots,c_n$ such that $\sum_{i=1}^n c_i = c+ n - 1$ we have that
$$g(E(\#_{i=1}^{n} K_{i})^{(c)}) \leq \sum_{i=1}^n g(E(K_{i})^{(c_{i})}) - (n-1)$$
and by Corollary~\ref{cor:sft-strong},  there exist integers $c_1,\dots,c_n$ for which
equality holds.   The corollary follows.
\end{proof}

\begin{proof}[Proof of Theorem~\ref{thm:sft}]
We induct on $(n,c)$ ordered lexicographically.  
Recall that in the beginning of the proof of Theorem~\ref{thm:strongHH} we showed that 
$(n,c)$ is well defined.  If $n = 1$ there is nothing to
prove; assume from now on $n > 1$.  

Assume Conclusion~(2) of Theorem~\ref{thm:sft} does not hold, that is, for
each $i$, $E(K_i)$ does not admit an essential meridional surface $S$ with $\chi (S)
\geq \genbd[g(\s)]$.   Then by the Swallow Follow Torus Theorem
\cite[Theorem~4.1]{kobayashi-rieck-m-small}
$\s$ weakly reduces to a swallow
follow torus, say $T$.  $T$ decomposes $E(\#_{i=1}^{n} K_i)^{(c)}$ as $E(K_I)^{(c_I)} \cup_T
E(K_J)^{(c_J)}$, where $I \subseteq \{1,\dots,n\}$ (possibly empty), $c_I + c_J
=c+1$, $K_I = \#_{i \in I} K_i$, and $K_J = \#_{i \not\in I} K_i$.  Denote
the Heegaard surfaces induced on  $E(K_I)^{(c_I)}$ and $E(K_J)^{(c_J)}$
by $\s_I$ and $\s_J$, respectively.

\bigskip
\noindent
{\bf Case One: $\emptyset \neq I \neq \{1,\dots,n\}$:}  In this case both 
$E(K_I)^{(c_I)}$ and $E(K_J)^{(c_J)}$ are exteriors of knots with strictly
less than $n$ prime factors and hence we may apply induction to both.
Since $g(\s_I) < g(\s)$, conclusion~(2) of Theorem~\ref{thm:sft} does not
hold for $E(K_I)^{(c_I)}$.
Hence, by induction, $\s_I$ admits  iterated weak reduction that yields a collection of
$|I| - 1$ swallow follow tori (say $\mathcal{T}_I \subset E(K_I)^{(c_I)}$) 
so that the following conditions hold: 
\begin{enumerate}
\item $\mathcal{T}_I$ decompose $E(K_I)^{(c_I)}$ as  $\cup_{\mathcal{T}_I} E(K_i)^{(c_i)}$ 
(for $i \in I$),
\item $\sum_{i \in I} c_i = c_{I} + |I| - 1$.  
\end{enumerate}
Similarly, $\s_J$ admits iterated weak reduction that yields a collection of
$(n-|I|) - 1$ swallow follow tori (say $\mathcal{T}_J \subset E(K_J)^{(c_J)}$) so that the following conditions hold: 
\begin{enumerate}
\item $\mathcal{T}_J$ decompose $E(K_J)^{(c_J)}$ as  $\cup_{\mathcal{T}_J} E_i^{(c_i)}$ (for $i \not\in I$),
\item $\sum_{i \not\in I} c_i = c_{J} + (n-|I|) - 1$.  
\end{enumerate}
Thus after iterated weak reduction of $\s$ we obtain 
$\mathcal{T} = T \cup \mathcal{T}_{I} \cup \mathcal{T}_{J}$.  By the above, $\mathcal{T}$
decomposes $E(\#_{i=1}^{n} K_{i})^{(c)}$ as $\cup_{\mathcal{T}} E(K_{i})^{(c_{i})}$, so that
$\sum_{i=1}^{n} c_{i} = \sum_{i \in I} c_{i} + \sum_{i \not\in I} c_{i} =  c_{I} + |I| - 1 + c_{J} + (n-|I|) - 1 = c + n -1$
(recall that $c_I + c_J = c+1$).
This proves Theorem~\ref{thm:sft} in Case One.

\bigskip
\noindent
{\bf Case Two: $I = \emptyset$ or  $I = \{1,\dots,n\}$.}  By symmetry we may assume that
$I = \{1,\dots,n\}$.  In that case $E(K_{J})^{(c_J)} \cong D(c_{J})$, a disk with $c_{J}$ holes cross
$S^{1}$, and $T$ gives the decomposition:
$$E(\#_{i=1}^{n} K_{i})^{(c)} = E(\#_{i=1}^{n} K_{i})^{(c_{I})} \cup_{T} D(c_{J})$$
Since $T$ is essential (and in particular, not boundary parallel) $c_{J} \geq 2$.
Since $c_{I} + c_{J} = c+1$, we have that $c_{I} < c$.  Thus the complexity of 
$E(\#_{i=1}^{n} K_{i})^{(c_{I})}$ is $(n,c_{I}) < (n,c)$ and we may apply induction to
$E(\#_{i=1}^{n} K_{i})^{(c_{I})}$.   Let $\s_{I}$ be the Heegaard surface for 
$E(\#_{i=1}^{n} K_{i})^{(c_{I})}$ induced by weak reduction.
By induction, $\s_{I}$ admits a  repeated weak reduction that yields
a system of $n-1$ swallow follow tori, say $\mathcal{T}_{I}$, that decomposes 
$E(\#_{i=1}^{n} K_{i})^{(c_{I})}$ as
$$E(\#_{i=1}^{n} K_{i})^{(c_{I})} = \cup_{\mathcal{T}_{I}} E(K_{i})^{(c_{i})}$$
with $\sum_{i=1}^{n} c_{i} = c_{I} + n -1$.  Let $T'$ be a component of 
$\mathcal{T}_{I}$.  
Then $T'$ decomposes
$E(\#_{i=1}^{n} K_{i})^{(c_{I})}$ as 
$E(\#_{i=1}^{n} K_{i})^{(c_{I})} = E(\#_{i \in I'} K_{i})^{(b_{1})}
\cup_{T'}  E(\#_{i \not\in I'} K_{i})^{(b_{2})}$, for some $I' \subseteq \{1,\dots,n\}$
and some integers $b_{1},b_{2} \geq 0$ with $b_{1} + b_{2} = c_I +1$.  
Since $T' \subset \mathcal{T}_{I}$, we have that $\emptyset \neq I' \neq \{1,\dots,n\}$.
By Proposition~\ref{pro:untel to a conn sep surfce implies weak reduction}, 
we see that $\s$ weakly reduce to $T'$.
This reduces Case Two to Case One, 
completing the proof of Theorem~\ref{thm:sft}.
\end{proof}

\section{Calculating the growth rate of m-small knots}
\label{sec:growth-rate}

In this final section we complete the proof of  Theorem~\ref{thm:main}.  
Let $K \subset M$ be
an m-small admissible knot in a compact manifold.  Recall the notation $nK$ and $E(K)^{(c)}$.

The difference between $g(E(K)^{(c)})$ and $g(E(K)) + c$ is measured by a
function denoted $\fff$ that plays a key role our work:

\begin{dfn}
\label{def:fk}
Given a knot $K$, we define the function 
$\fff : \mathbb{Z}_{\geq 0} \to \mathbb{Z}$  
to be
$$\fff(c) = g(E(K)) + c - g(E(K)^{(c)})$$
\end{dfn}
We immediately see that $\fff$ has the following properties, which we will often use without reference:
\begin{enumerate}
\item $\fff(0)= 0$.
\item For  $c \ge 0$, $\fff(c) \leq \fff(c+1) \leq \fff(c)+1$: this follows from the fact (proved in~\cite{rieck})
that for all $c \geq 0$, $g(E(K)^{(c)}) \leq g(E(K)^{(c+1)}) \leq g(E(K)^{(c)}) + 1$
\item For  $c \ge 0$, $0 \leq \fff(c) \leq c$ (follows easily from~(2)). 
\end{enumerate}
  
Before proceeding, we rephrase 
Corollaries~\ref{cor:sft-strong} and~\ref{cor:sft-strong2} in terms of $\fff$:

\begin{cor}
\label{cor:sft-strong3}
Let $K \subset M$ be a knot in a compact manifold and let $n$ be a positive integer.
Suppose that $E(K)$ does not admit a meridional essential surface $S$ with 
$\chi(S) \geq \genbd[g(E(nK))]$.  
Then there exist integers $c_1\dots,c_n \geq 0$ with $\s c_i=n-1$  so that: 
$$g(E(nK)) = n g(E(K)) - \sum_{i=1}^n \fff(c_i)$$
\end{cor}

\begin{proof}
By Corollary~\ref{cor:sft-strong} (with $c=0$) there exist $c_1,\dots,c_n \ge 0$ with 
$\s c_i = n-1$, so that $g(E(nK)) = \sum_{i=1}^n g(E(K)^{(c_{i})}) - (n-1)$.    We get:

\begin{eqnarray*}
  g(E(nK))  &=& \Big[ \sum_{i=1}^n g(E(K)^{(c_i)})\Big] -(n-1)  \\
          &=& \Big[\sum_{i=1}^n g(E(K)) + c_i - \fff(c_i)\Big]-(n-1)   \\
          &=& n g(E(K)) + \Big[\sum_{i=1}^n c_i\Big] - \Big[\sum_{i=1}^n \fff(c_i)\Big]-(n-1)\\
          &=& n g(E(K)) + (n-1) - \Big[\sum_{i=1}^n \fff(c_i)\Big]-(n-1)\\
          &=& n g(E(K)) - \sum_{i=1}^n \fff(c_i)
\end{eqnarray*}
\end{proof}

A similar argument shows that Corollary~\ref{cor:sft-strong2} gives:

\begin{cor}
\label{cor:sft-strong4}
Let $K \subset M$ be a knot in a compact manifold and let $n$ be a positive integer.
Suppose that $E(K)$ does not admit a meridional essential surface $S$ with 
$\chi(S) \geq \genbd[g(E(nK))]$.  
Then we have:
\begin{eqnarray*}
g(E(nK)) &=& \min\big\{ng(E(K)) - \s_{i=1}^n \fff(c_i)\big\} \\
       &=& ng(E(K)) - \max \big\{ \s_{i=1}^n \fff(c_i)\big\} 
\end{eqnarray*}
where the minimum and maximum are taken over all integers $c_1,\dots,c_n \geq 0$ with $\sum_{i=1}^{n} c_i=n-1$.
\end{cor}

Recall (Notation~\ref{notation:bridge indices}) that we denote $g(E(K)) - g(M)$ by $g$
and the bridge indices of $K$ with respect to
Heegaard surfaces of genus $g(E(K)) - i$
by $b_i^*$ ($=1,\dots,g$), so that
$0<b^{*}_1<\dots<b^{*}_i<\dots<b^{*}_g$.
We formally
set $b_0^* = 0$ and $b_{g+1}^* = \infty$.  
Note that these properties imply that for every $c \geq 0$ there is a unique index $i$ ($0 \leq i \leq g$),
depending on $c$, 
so that $b^{*}_i \leq c  < b^{*}_{i+1}$; we will use this fact below without reference.

In the following proposition we calculate $\fff(c)$  
when $E(K)$ does not admit an essential
meridional surface $S$ with $\chi(S) \geq \genbd[g(E(K)^{(c)})]$.

\begin{pro}
\label{pro:fk}
Let $K$ be a knot and $c \geq 0$ an integer.  Let $0 \leq i \leq g$ be the unique index for which
$b^{*}_i \leq c  < b^{*}_{i+1}$.  Then $\fff(c) \geq i$.  If in addition $E(K)$ does
not admit an essential meridional surface $S$ with 
$\chi(S) \geq \genbd[g(E(K)^{(c)})]$ then equality holds:
$$\fff(c) = i$$
\end{pro}

\begin{proof}[Proof of Proposition~\ref{pro:fk}]
We first prove that $\fff(c) \geq i$ holds for any knot.  Since $\fff$ is a non-negative 
function we may assume $i \geq 1$.
By the definition of $b^{*}_i$, $K$ admits a ($g(E(K)) - i$, $b^{*}_i$) decomposition.  
Since $c \geq b^{*}_i$, $K$ admits a $(g(E(K)) - i, c)$ decomposition.  By~Corollary~\ref{cor:upper-bound-for-X(b)}
we have that $g(E(K)^{(c)}) \leq g(E(K)) -i + c$.
Therefore, $\fff(c) = g(E(K)) + c - g(E(K)^{(c)}) \geq g(E(K)) + c - (g(E(K)) -i+c) = i$.

Next we assume, in addition, that $E(K)$ does not admit an essential meridional
surface $S$ with $\chi(S) \geq \genbd[g(E(K)^{(c)})]$.  
We will complete the proof of the proposition by showing that
$\fff(c) < i+1$; suppose for a contradiction that $\fff(c) \geq i+1$.
Thus $g(E(K)^{(c)}) = g(E(K)) + c - \fff(c) \leq g(E(K)) + c - (i+1)$.

Assume first that $i=g$.  Then by 
Corollary~\ref{cor:strongHH} (with  $g(E(K))+c-(g+1)$ corresponding to 
$h$) we see that 
$k$ admits a $(g(E(K))+c-(g+1)-c, c)$ decomposition.  In particular, $M$ 
admits a Heegaard surface of genus $(g(E(K)))+c-(g+1)-c$.  Hence we see:
\begin{eqnarray*}
g(M) & \leq & (g(E(K))+c-(g+1)-c \\
& = & g(E(K)) - g-1 \\
& = & g(E(K)) - (g(E(K)) - g(M)) -1 \\
&=& g(M) -1
\end{eqnarray*}
This contradiction completes the proof when $i=g$.

Next assume that $0 \le i < g$.  Applying
Corollary~\ref{cor:strongHH} again (with $g(E(K))+c-(i+1)$ corresponding to 
$h$ in Corollary~\ref{cor:strongHH}) we see that  $K$ admits a
$(g(E(K)) - (i+1), c)$ decomposition.  By definition, $b^{*}_{i+1}$ is the smallest integer 
so that $K$ admits a $(g(E(K)) - (i+1),b^{*}_{i+1})$ decomposition; hence $c \geq b^{*}_{i+1}$.
This contradicts our choice of $i$ in the statement of the proposition, showing that $\fff(c) < i+1$.  
This completes the proof of Proposition~\ref{pro:fk}.
\end{proof}

As an illustration of Proposition~\ref{pro:fk}, let $K$ be an m-small knot in $S^3$. 
Suppose that $g=3$,   $b^{*}_1=5$, $b^{*}_2=7$, and $b^{*}_3 =
23$. (We do not know if a knot with these properties exists.) Then:

$$\fff(c) =   \left\{
\begin{array}{l} 
0 \hspace{1in} 0 \leq c \leq 4 \\  %
1 \hspace{1in} 5 \leq c \leq 6 \\ %
2 \hspace{1in} 7 \leq c \leq 22 \\ %
3 \hspace{1in} 23 \leq c \\ %
\end{array} \right.$$

Not much is known about $\fff$ for knots that are not m-small:

\begin{question}
\label{que:bounded}
Does there exist a knot $K$ in a manifold $M$ with unbounded $\fff$?
Does there exist a knot $K$ with $\fff(c) > g(E(K)) - g(M)$ (for sufficiently large $c$)?
What can be said about the behavior of the function $\fff$? 
\end{question}

With the preparation complete, we are now ready to prove Theorem~\ref{thm:main}.

\begin{proof}[Proof of Theorem~\ref{thm:main}]
Fix the notation of Theorem~\ref{thm:main}.  Since the upper bound was obtained in
Proposition~\ref{pro:upper-bound}, we assume from now on that $K$ is m-small.
By Corollary~\ref{cor:sft-strong4}, $g(E(nK)) = n g(E(K)) - \max \{ \s_{i=1}^n
\fff(c_i) \}$, where the maximum is taken over all integers $c_{1},\dots,c_{n} \geq 0$ 
with $\sum_{i=1} ^{n} c_i = n-1$.

Fix $n$ and let $c_1,\dots,c_n \geq 0$ be integers with $\sum_{i=1}^{n} c_i = n-1$ that maximize
$\s_{i=1}^n \fff(c_i)$.

\begin{lem}
\label{lem:maximizingsequence}
We may assume that the sequence $c_1,\dots,c_{n}$ fulfills the following conditions for some $1 \leq l \leq n$:
\begin{enumerate}
\item $c_{i} \geq c_{i+1}$ ($i=1,\dots,n-1$).
\item For $i \leq l$, $c_{i} \in \{b_{1}^{*},\dots,b_{g}^{*}\}$.
\item  $c_{l+1} < b_{1}^{*}$.
\item For $i > l+1$, $c_{i} = 0$.
\end{enumerate}
\end{lem}

\begin{proof}
By reordering the indices if necessary we 
may assume~(1) holds.

Let $l$ be the largest index for which $\fff(c_l) \neq 0$.  
For $i=1,\dots,l$, let $0 \leq j(i) \leq g$ be the unique index for which 
$b^{*}_{j(i)} \le c_{i} < b^{*}_{j(i)+1}$ (recall that we set $b^{*}_0 =0$ and $b_{g+1}^{*} = \infty$).
Define $c'_{1},\dots,c'_{n}$ as follows:

\begin{enumerate}
\item For $i  \leq l$, set $c'_i =b^{*}_{j(i)}$ (in other words, $c_{i}'$ is the largest $b_{j}^{*}$
that does not exceed $c_{i}$).
\item Set $c'_{l+1} = n-1- (\sum _{i=1} ^{l} c'_i)$.
\item For $i > l+1$ set $c'_{i} = 0$.
\end{enumerate}

By Proposition~\ref{pro:fk}, for $i \leq l$, $\fff(c_{i}) = \fff(b^{*}_{j(i)}) = \fff(c_{i}')$.  
We get :

\begin{eqnarray*}
\sum_{i=1}^{n} \fff(c'_{i}) &=& \sum_{i=1}^{l} \fff(c'_{i}) + \sum_{i=l+1}^{n} \fff(c'_{i}) \\
&=& \sum_{i=1}^{l} \fff(c_{i}) + \sum_{i=l+1}^{n} \fff(c'_{i}) \\
&\geq& \sum_{i=1}^{l} \fff(c_{i}) \\
&=& \sum_{i=1}^{n} \fff(c_{i}).
\end{eqnarray*}

(For the last equality, recall that $\fff(c_{i}) = 0$ for $i > l$.)

Since $c_{1},\dots,c_{n}$ maximizes $\sum_{i=1}^{n} \fff(c_{i})$,
we conclude that $\sum_{i=1}^{n} \fff(c_{i}) = \sum_{i=1}^{n} \fff(c_{i}')$ and 
hence $\fff(c'_{l+1}) = 0$; thus $c_{l+1}' < b_1^*$.
Thus $c'_{1},\dots,c'_{n}$ is a maximizing sequence;
it is easy to see that it fulfills conditions~(1)--(4).
\end{proof}

We will denote the $n^{\rm th}$
term of the defining sequence of the growth rate by $S_{n}$, that is:
$$S_{n} =  \frac{g(E(nK)) - ng(E(K)) + n-1}{n-1}$$  
By Corollary~\ref{cor:sft-strong4} the following holds:

\begin{equation}
\label{eq:S_n}
S_{n}     = 1 - \frac{\max \{ \Sigma_{i=1}^n \fff(c_i) \} }{n-1}
\end{equation}
In order to bound $S_{n}$ below we need to understand the following optimization problem, where
here we are assuming that the maximizing sequence fulfills the conditions listed in  Lemma~\ref{lem:maximizingsequence},
and in particular, $\fff(c_{i})=0$ for $i > l$. 

\begin{prob}\label{prob:original}
Find non negative integers $l$ and $c_1,\dots,c_l$ that 
maximize $\sum_{i=1}^l \fff (c_i)$ subject to the constraints:
\begin{enumerate}
%\item $\sum_{i=1}^{l} c_i  = n-1 - c_{l+1}$, where $0 \leq  c_{l+1} < b^{*}_{1}$
%\item $n-1-b_1^* < \sum_{i=1}^l c_i \leq n-1$
\item $\sum_{i=1}^l c_i \leq n-1$
\item $c_{i}  \in \{b^{*}_1,\dots,b^{*}_g\}$ (for $1 \leq i \leq l$).
\end{enumerate}
\end{prob}

For $i=1,\dots,g$, let $k_i$ be the number of times that $b^{*}_i$ appears in
$c_1,\dots,c_l$.  By Proposition~\ref{pro:fk}, $\fff(b^{*}_i) = i$; thus 
Problem~\ref{prob:original} can be rephrased as follows:  

\begin{prob}\label{prob:rewritten}
Maximize $\sum_{i=1}^g k_i i$ subject to the constraints:
\begin{enumerate}
%\item $\sum_{i=1}^g k_i b^{*}_i = n - 1 - c_{l+1}$, where $0 \leq  c_{l+1} < b^{*}_{1}$
%\item $n-1-b_i^* < \sum_{i=1}^g k_i b^{*}_i \leq n-1$
\item $\sum_{i=1}^g k_i b^{*}_i \leq n-1$
\item $k_{i}$ is a non-negative integer
\end{enumerate}
\end{prob}

We first solve this optimization problem over $\mathbb{R}$;
%it is easy to see that over $\mathbb{R}$ the solution will be
%with $c_{l+1} = 0$.  
we use the variables $x_1, \dots, x_g$ instead of $k_1, \dots, k_g$.

\begin{prob}\label{prob:over-R}
Given $n \in \mathbb{R}$, $n > 1$, maximize $\sum_{i=1}^g x_{i} i$ 
subject to the constraints  
\begin{enumerate}
\item $\sum_{i=1}^g x_i b^{*}_i \leq n-1$
\item $x_1 \ge 0, \dots, x_g \ge 0$
\end{enumerate}

\end{prob}

It is easy to see that for any sequence $x_1,\dots,x_g$ that realizes maximum
we have that $\sum_{i=1}^g x_i b^{*}_i = n-1$, for otherwise we can increase the value of
$x_1$, thus increasing $\sum_{i=1}^g x_{i} i$ and contradicting maximality.
Problem~\ref{prob:over-R} is an elementary linear programming problem
(known as the standard maximum problem) and is solved using the
simplex method which gives: 

\begin{lem}
\label{lem:over-R}
There is a (not necessarily unique) index $i_0$, 
which is independent of $n$, such that a solution of Problem~\ref{prob:over-R} 
is given by% setting $x_i = 0$ (for $i \ne i_0$) and
$$x_{i_0} = \frac{n-1}{b_{i_0}^*}, \ \ \ \ x_i = 0 (i \neq i_0) $$ 
Hence the maximum is 
$$\frac{(n-1)i_0}{b_{i_0}^*}$$   
\end{lem}

\begin{proof}[Proof of Lemma~\ref{lem:over-R}]
The notation used in this
proof was chosen to be consistent with notation often used in linear
programming texts.
Let ${\bf \overrightarrow N},{\bf \overrightarrow F}$ and $\vec{x} \in \mathbb{R}^{g}$
denote the following vectors 
$${\bf \overrightarrow N} = (b_{1}^{*},\dots,b_{g}^{*}), \ \ \ {\bf \overrightarrow F} = (1,\dots,g), \ \mbox{ and } \ \vec{x} = (x_{1},\dots,x_{g})$$
For $n \in \mathbb{R}$, $n>1$, let $\Delta_{n}$ be 
$$\Delta_{n} = \{\vec{x} \in \mathbb{R}^{g} \ | {\bf \overrightarrow N} \cdot \vec{x} = n-1, x_{1} \geq 0,\dots,x_g \geq 0 \}$$
Note that $\Delta_{n}$ is a simplex and its codimension $k$ faces are obtained by setting
$k$ variables to zero.
Problem~\ref{prob:over-R} can be stated as:
$$\mbox{maximize   } \ \ {\bf \overrightarrow F}\cdot \vec{x}, \ \ \mbox{   subject to   } \ \ \vec{x} \in \Delta_{n}$$
Since the gradient of ${\bf \overrightarrow F}\cdot \vec{x}$
is ${\bf \overrightarrow F}$ and the normal to $\Delta_{n}$
is ${\bf \overrightarrow N}$, the gradient of the restriction of
${\bf \overrightarrow F}\cdot \vec{x}$ to ${\Delta_{n}}$ is
the projection
$${\bf \overrightarrow P} = {\bf \overrightarrow F} - 
\frac{{\bf \overrightarrow F} \cdot {\bf \overrightarrow N}}{|{\bf \overrightarrow N}|^{2}}{\bf \overrightarrow N}$$
Note that ${\bf \overrightarrow P}$ is independent of $n$.  The maximum of
${\bf \overrightarrow N}\cdot \vec{x}$ on ${\Delta_{n}}$ is found by moving along 
$\Delta_{n}$ in the direction of ${\bf \overrightarrow P}$.  This
shows that the maximum is obtained along a face defined by setting
some of the variables to zero, and the variables set to zero are independent
of $n$.  Lemma~\ref{lem:over-R} follows by picking $i_{0}$
to be one of the variables not set to zero.
\end{proof}

Fix an index $i_{0}$ as in Lemma~\ref{lem:over-R}.
If $b^{*}_{i_{0}} | n-1$ then the maximum (over $\mathbb{R}$) found in  Lemma~\ref{lem:over-R}
is in fact an integer and hence is also the maximum for Problem~\ref{prob:original}.  
This allows us to calculate $S_{n}$ in this case:

\begin{lem}
\label{lem:divisible}
If $b^{*}_{i_{0}} | n-1$ then $S_{n} =  1 - {i_{0}}/{b^{*}_{i_{0}}}$.
\end{lem}

\begin{proof}
$$S_{n}     = 1 - \frac{\max \{ \Sigma_{i=1}^n \fff(c_i) \} }{n-1} = 1 - \frac{(n-1)i_0}{(n-1)b_{i_0}^*} =  1 - \frac{i_0}{b_{i_0}^*}$$
\end{proof}

We now turn our attention to the general case, where $b^{*}_{i_{0}}$ may not divide $n-1$.
We will only consider values of $n$ for which $n > b_{i_{0}}^{*}$.
As in Section~\ref{sec:upper-bound}, let $k_{i_{0}}$ and $r$ be the quotient and remainder
when dividing $n-1$ by $b_{i_{0}}^{*}$, so that
\begin{equation}
\label{eq:Dividing}
n-1 = k_{i_{0}}b_{i_{0}}^{*} + r, \ \ \ \\ \ 0 \leq r < b_{i_0}^{*}
\end{equation}
Let $c_j \geq 0$ ($1 \leq j \leq n$) be integers with $\sum_{j=1}^{n} c_{j}= n - 1$ that maximize 
$\sum_{j=1}^{n} \fff(c_j)$.
We will denote $n-r$ by $n'$.  
Let $c_j' \geq 0$ ($1 \leq j \leq n'$) be integers with $\sum_{j=1}^{n'} c'_{j}= n' - 1$ that maximize 
$\sum_{j=1}^{n'} \fff(c'_j)$.

\begin{clm}
\label{clm:sum}
$\sum_{j=1}^{n} \fff(c_j) \leq \sum_{j=1}^{n'} \fff(c'_j) + r$.
\end{clm}

\begin{proof}
Starting with the sequence $c_1,\dots,c_n$, we obtain a new sequence by subtracting one from exactly one
$c_j$ (with $c_j > 0$).  Let $c_{j}'''$ be a sequence of non negative integers obtained by repeating this process $r$
times.  Then $\sum_{j=1}^{n} c_{j}''' = n-1-r = n'-1$.  Let $c_{j}''$
be the sequence obtained from $c_{j}'''$ by removing $r$ zeros (note that this is possible as there indeed are at least $r$ zeros).
We get:
\begin{align*}
\sum_{j=1}^{n'} \fff(c'_j) + r &\geq \sum_{j=1}^{n'} \fff(c''_{j}) + r  & \mbox{ since } c_{j}'  \mbox{ maximizes}\\
&= \sum_{j=1}^{n} \fff(c'''_{j}) + r & \mbox{ since } \fff(0) = 0 \\
&\geq \sum_{j=1}^{n} \fff(c_{j}) & \mbox{ since } \fff(c) + 1 \geq \fff(c+1)
\end{align*}

This proves Claim~\ref{clm:sum}.
\end{proof}

Note that $b_{i_{0}}^{*} | n'-1$ and so we may apply Lemma~\ref{lem:divisible}
to calculate $S_{n'}$.  We get (in the first line we use Equation~(\ref{eq:S_n}) from Page~\pageref{eq:S_n}):

\begin{align*}
S_{n} &=  1 - \frac{\max \{ \Sigma_{i=1}^n f(c_i) \} }{n-1} & \mbox{Equation}~(\ref{eq:S_n})\mbox{ for } S_{n} \\
&\geq  1 - \frac{\max \{ \Sigma_{j=1}^{n'} f(c'_j) + r \}}{n-1} &  \mbox{Claim~}\ref{clm:sum} \\
&=  1 - \frac{n'-1}{n-1}\frac{\max \{ \Sigma_{j=1}^{n'} f(c'_j) \} }{n'-1} -\frac{r}{n-1} & \\
&=  \frac{n'-1}{n-1}\Big(1 - \frac{\max \{ \Sigma_{j=1}^{n'} f(c'_j) \} }{n'-1}\Big)  + \Big(1 - \frac{n'-1}{n-1}\Big) - \frac{r}{n-1} & \\
&= \frac{n'-1}{n-1}S_{n'} + \Big(1 - \frac{n'-1}{n-1}\Big) - \frac{r}{n-1} & \mbox{Equation}~(\ref{eq:S_n})\mbox{ for } S_{n'} \\ 
&=  \frac{n'-1}{n-1}\Big(1 - \frac{i_{0}}{b^{*}_{i_{0}}}\Big) +  \Big(1 - \frac{n'-1}{n-1}\Big) - \frac{r}{n-1}& \mbox{Lemma}~\ref{lem:divisible} \\
&=  \frac{n'-1}{n-1}\Big(1 - \frac{i_{0}}{b^{*}_{i_{0}}}\Big) +  \Big(1 - \frac{n'+r-1}{n-1}\Big) & \\
&=  \frac{n-r-1}{n-1}\Big(1 - \frac{i_{0}}{b^{*}_{i_{0}}}\Big) & \mbox{Substituting } n' = n-r
\end{align*}
Recall that in the proof of Proposition~\ref{pro:upper-bound} (see Page~\pageref{equ:UpperBound})
we proved Equation~(\ref{equ:UpperBound}) which says (recall that $k_{i_0}$ was defined in Equation~(\ref{eq:Dividing})
above):
$$S_n < 1 - \frac{i_{0}}{b_{i_{0}}^*} \frac{k_{i_{0}}}{k_{i_{0}}+1}$$
Combining these facts we obtain:  
$$  \frac{n-r-1}{n-1}\Big(1 - \frac{i_{0}}{b^{*}_{i_{0}}}\Big)  \leq S_{n}  < 1 - \frac{i_{0}}{b_{i_{0}}^*} \frac{k_{i_{0}}}{k_{i_{0}}+1}$$
By Equation~(\ref{eq:Dividing}) above, $r < b^{*}_{i_{0}}$ and $\lim_{n \to \infty} k_{i_{0}} = \infty$.  
We conclude that as $n \to \infty$ both bounds limit on $1 - i_{0}/b^{*}_{i_{0}}$,
and thus $\lim_{n \to \infty} S_{n}$
exists and equals $1 - i_{0}/b^{*}_{i_{0}}$.

This completes the proof of Theorem~\ref{thm:main}.
\end{proof}

% ----------------------------------------------------------------


\begin{thebibliography}{10}


\bibitem{BakerKobayashiRieck}
Kenneth L Baker, Tsuyoshi Kobayashi, and Yo'av Rieck
\newblock  The spectrum of the growth rate of the tunnel number is infinite.
\newblock Preprint, 2015.



\bibitem{BowmanTaylorZupan}
R. Sean Bowman, Scott A. Taylor, and A. Zupan. 
\newblock  Bridge spectra of twisted torus knots.
\newblock {\em Int. Math. Res. Notices}, doi: 10.1093/imrn/rnu162, 2014.






\bibitem{casson-gordon}
A.~J. Casson and C.~McA. Gordon.
\newblock Reducing {H}eegaard splittings.
\newblock {\em Topology Appl.}, 27(3):275--283, 1987.

\bibitem{haken}
Wolfgang Haken.
\newblock Some results on surfaces in {$3$}-manifolds.
\newblock In {\em Studies in Modern Topology}, pages 39--98. Math. Assoc. Amer.
  (distributed by Prentice-Hall, Englewood Cliffs, N.J.), 1968.

\bibitem{hempel}
John Hempel.
\newblock {\em {$3$}-{M}anifolds}.
\newblock Princeton University Press, Princeton, N. J., 1976.
\newblock Ann. of Math. Studies, No. 86.





\bibitem{IchiharaSaito}
Ichihara, Kazuhiro and Saito, Toshio.
\newblock Knots with arbitrarily high distance bridge decompositions.
\newblock {\em Bull. Korean Math. Soc.},  50 (2013), no. 6, 1989--2000. 







\bibitem{jaco}
William Jaco.
\newblock {\em Lectures on three-manifold topology}, volume~43 of {\em CBMS
  Regional Conference Series in Mathematics}
\newblock American Mathematical Society, Providence, R.I., 1980.

\bibitem{jaco-shalen}
W. Jaco, and P.Shalen. 
\newblock{\em Seifert fibered spaces in 3-manifolds}, 
\newblock {Mem. Amer. Math. Soc., 21 (1979), no.220}


\bibitem{Johannson}
K. Johannson,
\newblock Homotopy equivalences of 3-manifolds, 
\newblock Lecture Notes in Mathematics 761. Springer, (1979).


\bibitem{kim}
Kim, Soo Hwan.
\newblock The tunnel number one knot with bridge number three is a $(1,1)$-knot.
\newblock {\em Kyungpook Math. J.}, 45(1):67--71, 2005.
		

\bibitem{kobayashi-rieck-local-det}
Tsuyoshi Kobayashi and Yo'av Rieck.
\newblock Local detection of strongly irreducible {H}eegaard splittings via
  knot exteriors.
\newblock {\em Topology Appl.}, 138(1-3):239--251, 2004.

\bibitem{kobayashi-rieck-growth-rate}
Tsuyoshi Kobayashi and Yo'av Rieck.
\newblock {On the growth rate of tunnel number of knots},   {\it Journal f\"ur
  die reine und angewandte Mathematik} 592 (2006) 63--78. 

\bibitem{kobayashi-rieck-m-small}
Tsuyoshi Kobayashi and Yo'av Rieck.
\newblock {Heegaard genus of the connected sum of ${m}$-small knots}, 
{\it Comm. Anal. Geom.} 5 (2006) 1037--1077.

\bibitem{KobayashiRieckAnnouncement}
Tsuyoshi Kobayashi and Yo'av Rieck.
\newblock Knots with $g(E(K)) = 2$ and $g(E(K\# K\# K)) = 6$ and Morimoto's Conjecture, 
\newblock {\em Topology and its Appl. }, 156(2009) 1114--1117.



\bibitem{KobayashiRieckMC}
Tsuyoshi Kobayashi and Yo'av Rieck.
\newblock Knot exteriors with additive Heegaard genus and Morimoto's conjecture.
\newblock {\em Algebr. Geom. Topol. }, 8 (2008), no. 2, 953--969.






\bibitem{KobayashiSaito}
T.Kobayashi, T.Saito.
\newblock Destabilizing Heegaard splittings of knot exteriors
\newblock {\em Topology and its Appl. }, 157(1):202--212, 2010.




\bibitem{miyazaki}
Katura Miyazaki.
\newblock Conjugation and the prime decomposition of knots in closed, oriented
  {$3$}-manifolds.
\newblock {\em Trans. Amer. Math. Soc.}, 313(2):785--804, 1989.

\bibitem{morimoto-proc-ams}
Kanji Morimoto.
\newblock There are knots whose tunnel numbers go down under connected sum.
\newblock {\em Proc. Amer. Math. Soc.}, 123(11):3527--3532, 1995.






\bibitem{morimoto3}
Kanji Morimoto.
\newblock Essential surfaces in the exteriors of torus knots with twists on 2-strands.
\newblock Available at http://morimoto.ii-konan.jp/TKSML.pdf.








\bibitem{skuma-morimoto-yokata}
Kanji Morimoto, Makoto Sakuma, and Yoshiyuki Yokota.
\newblock Examples of tunnel number one knots which have the property
  ``{$1+1=3$}''.
\newblock {\em Math. Proc. Cambridge Philos. Soc.}, 119(1):113--118, 1996.



\bibitem{norwood}
Norwood, F. H.
\newblock Every two-generator knot is prime. 
\newblock {\em Proc. Amer. Math. Soc.}, 86(1):143--147, 1982.





  

\bibitem{rieck}
Yo'av Rieck.
\newblock Heegaard structures of manifolds in the {D}ehn filling space.
\newblock {\em Topology}, 39(3):619--641, 2000.


\bibitem{no-nesting}
Martin Scharlemann.
\newblock Local detection of strongly irreducible {H}eegaard splittings.
\newblock {\em Topology Appl.}, 90(1-3):135--147, 1998.

\bibitem{scharlemann-handbook}
Martin Scharlemann.
\newblock Heegaard splittings of compact 3-manifolds.
\newblock In {\em Handbook of geometric topology}, pages 921--953.
  North-Holland, Amsterdam, 2002.


\bibitem{scharl-abby}
Martin Scharlemann and Abigail Thompson.
\newblock Thin position for {$3$}-manifolds.
\newblock In {\em Geometric topology (Haifa, 1992)}, volume 164 of {\em
  Contemp. Math.}, pages 231--238. Amer. Math. Soc., Providence, RI, 1994.

\bibitem{schultens-FXS1}
Jennifer Schultens.
\newblock The classification of {H}eegaard splittings for (compact orientable
  surface){$\,\times\, S\sp 1$}.
\newblock {\em Proc. London Math. Soc. (3)}, 67(2):425--448, 1993.


\bibitem{schultens-graph}
Jennifer Schultens.
\newblock {H}eegaard splittings of graph manifolds.
\newblock {\em  Geom. Topol.} 8, 831--876 (electronic), 2004.






\bibitem{Zupan}
A. Zupan.
\newblock Bridge spectra of iterated torus knots.
\newblock {\em Comm. Anal. Geom.} 22 (2014), no. 5, 931--963


\end{thebibliography}
\end{document}